\documentclass{amsart}

\usepackage{amssymb}
\usepackage{amsthm}
\usepackage{amsfonts}
\usepackage{mathrsfs}
\usepackage{amsmath}

\numberwithin{equation}{section}

\newtheorem{thm}{Theorem}{\bf}{\it}
\newtheorem{cor}[thm]{Corollary}

\newtheorem{lem}[thm]{Lemma}
\newtheorem{conj}[thm]{Conjecture}

\newcommand{\hyperg}[4]{\: _2\! F_1\! \left[ \begin{array}{c} #1,\, #2 \\ #3 \end{array} ;\,  #4 \right]}
\newcommand{\RePt}{\mathrm{Re}\,}
\newcommand{\ImPt}{\mathrm{Im}\,}
\newcommand{\ball}{\mathbb{B}_{n+1}}
\newcommand{\sphere}{\mathbb{S}_{n+1}}
\newcommand{\disk}{\mathbb{D}}

\newcommand{\calU}{\mathcal{U}}

\newcommand{\bfi}{\mathbf{i}}
\newcommand{\bfbeta}{\boldsymbol{\beta}}
\newcommand{\bfrho}{\boldsymbol{\rho}}

\newcommand{\bbB}{\mathbb{B}}
\newcommand{\bbC}{\mathbb{C}}
\newcommand{\bbR}{\mathbb{R}}
\newcommand{\bbH}{\mathbb{H}}

\begin{document}

\title[The Bergman and Cauchy-Szeg\"o projections]{Norm estimates for the Bergman and Cauchy-Szeg\"o projections
over the Siegel upper half-space}

\thanks{This work was supported by the National Natural Science
Foundation of China grants 11571333, 11471301.}

\author{Congwen Liu}
\email{cwliu@ustc.edu.cn}

\address{School of Mathematical Sciences,
University of Science and Technology of China,\\
Hefei, Anhui 230026,
People's Republic of China\\
and\\
Wu Wen-Tsun Key Laboratory of Mathematics\\
USTC, Chinese Academy of Sciences}


\begin{abstract}
We obtain several estimates for the $L^p$ operator norms of the Bergman
and Cauchy-Szeg\"o projections over the the Siegel upper half-space.
As a by-product, we also determine the precise value of the $L^p$ operator norm of a
family of integral operators over the Siegel upper half-space.
\end{abstract}


\maketitle

\section{Introduction}

Let $\calU^n$ be the Siegel upper half-space (or the generalized half-plane, following the terminology
of Kor\'anyi \cite{Kor65,KS68,KV71,KW65})
\[
\calU^{n} := \left\{ z\in \bbC^{n+1}: \ImPt z_{n+1} > |z^{\prime}|^2 \right\}
\]
and let $b\calU^n$ be its boundary in $\bbC^{n+1}$. Here and throughout, we use the notation
\[
z=(z^{\prime},z_{n+1}), \quad \text{where } z^{\prime}=(z_1,\ldots,z_n)\in \bbC^n \text{ and } z_{n+1}\in \bbC^1.
\]
Note that $\calU^0=\bbC_{+}:=\{z\in \bbC: \ImPt z>0\}$, the classical upper half-plane.
$\calU^n$ is biholomorphically equivalent to the unit ball $\ball$ in
$\bbC^{n+1}$, via the Cayley transform $\Phi:\bbB_{n+1}\to \calU^n$ given by
\[
(z^{\prime}, z_{n+1})\; \longmapsto\; \left( \frac {z^{\prime}}{1+z_{n+1}},
i\frac {1-z_{n+1}}{1+z_{n+1}} \right),
\]
and so it is also referred to as the unbounded realization of the unit ball in $\bbC^{n+1}$.
%$\calU^n$ is the image of the complex unit ball $\bbB^{n+1}$ under the Cayley transformation $\Phi$ defined by
%\[
%\Phi:\ (z^{\prime}, z_{n+1})\; \longmapsto\; \left( \frac {z^{\prime}}{1+z_{n+1}},
%i\frac {1-z_{n+1}}{1+z_{n+1}} \right).
%\]
%Indeed, $\Phi$ maps $\bbB^{n+1}$ to $\calU^n$ in a holomorphic, one-to-one, and onto fashion, and
%hence is a biholomorphic isomorphism.
%It will be convenient to use the notation
%\[
%z=(z^{\prime},z_{n+1}), \quad \text{where } z^{\prime}=(z_1,\ldots,z_n)\in \bbC^n \text{ and } z_{n+1}\in \bbC^1,
%\]
%and let $\bfrho(z):=\ImPt z_{n+1}-|z^{\prime}|^2$ (which is usually called the height function in $\calU^n$).
%In this notation,
%\[
%\calU^n=\{z\in \bbC^{n+1}: \bfrho(z)>0\}\; \text{ and } \; b\calU^n=\{z\in \bbC^{n+1}: \bfrho(z)=0\}.
%\]

As usual, for $p>0$, the space $L^p(\calU^n)$ consists of all Lebesgue measurable
functions $f$ on $\calU^n$ for which
\begin{equation*}
\|f\|_p:=\bigg\{\int\limits_{\calU^n} |f(z)|^p dV(z)\bigg\}^{1/p}
\end{equation*}
is finite, where $dV=dm_{2n+2}$ is the Lebesgue measure on $\bbC^{n+1}$.
The Bergman space $A^p(\calU^n)$ is the closed subspace of $L^p(\calU^n)$ consisting of
holomorphic functions on $\calU^n$.
The orthogonal projection from $L^2(\calU^n)$ onto $A^2(\calU^n)$, known as the Bergman projection,
can be expressed as an integral operator:
\[
(P_{\calU^n}f)(z) = \int\limits_{\calU^n} K_{\calU^n}(z,w) f(w) dV(w),
\]
with the Bergman kernel
\begin{equation}\label{eqn:bergknl}
K_{\calU^n}(z,w)= \frac {(n+1)!}{4\pi^{n+1}}\,\left[\frac {i}{2} (\bar{w}_{n+1}-z_{n+1})
- z^{\prime} \cdot \overline{w^{\prime}} \right]^{-n-2}.
\end{equation}
See \cite[p.56, Lemma 5.1]{Gin64}. % or \cite[p.81 (3.11)]{PS86}.
In the sequel, we shall use the notation
\begin{equation*}
\bfrho(z,w) ~:=~ \frac {i}{2} (\bar{w}_{n+1}-z_{n+1})
- z^{\prime} \cdot \overline{w^{\prime}}.
\end{equation*}
It has been long known that the Bergman projection $P_{\calU^n}$ extends to a
bounded operator from $L^p(\calU^n)$ to $A^p(\calU^n)$, for $1<p<\infty$. See, for instance, \cite[Lemma 2.8]{CR80}.

%For $z\in \calU^n$, the Bergman reproducing kernel is the
%function $K_z\in A^2(\calU^n)$ such that $f(z)=\langle f, K_z\rangle$ for every
%$f\in A^2(\calU^n)$, where $\langle \cdot,\cdot\rangle$ denotes the inner product in $L^2(\calU^n)$.
%Explicitly (see, for instance, \cite[p.81 (3.11)]{PS86}),
%\begin{equation*}
%K_z(w)=\frac {(n+1)!}{4\pi^{n+1}}\, \bfrho(z,w)^{-n-2},\quad z\in \calU^n,
%\end{equation*}
%where
%\[
%\bfrho(z,w)=\frac {i}{2} (\bar{w}_{n+1}-z_{n+1}) - z^{\prime} \cdot \overline{w^{\prime}}.
%\]
%Here, as usual, $\xi \cdot \overline{\eta} := \sum_{j=1}^n \xi_j\overline{\eta}_j$ denotes
%the Hermitian inner product of $\xi,\eta\in \bbC^n$.
%The normalized reproducing kernel $k_z$ is defined by $k_z=K_z/\|K_z\|_2$.
%

In this paper, we are concerned with estimates of the operator norm of $P_{\calU^n}$ on
$L^p(\calU^n)$. Our first main result is the following.

\begin{thm}\label{thm:main0}
For $1<p<\infty$, we have
\begin{equation}\label{eqn:bergmanestimate}
\frac {\Gamma\left(\frac {n+2}{p}\right) \Gamma\left(\frac {n+2}{q}\right)}{
 \Gamma^2 \left(\frac {n+2}{2}\right)} ~\leq~ \|P_{\calU^n}\|_{p\to p}
 ~\leq~ \frac {(n+1)!}{\Gamma^2\left(\frac {n+2}{2}\right)} \frac {\pi}{\sin \frac {\pi}{p}},
\end{equation}
where $q:=\frac {p}{p-1}$ is the conjugate exponent of $p$ and
\[
\|P_{\calU^{n}}\|_{p\to p} := \|P_{\calU^{n}}\|_{L^p(\calU^{n})\to A^p(\calU^{n})}
= \sup \left\{ \frac {\|P_{\calU^{n}}f\|_p} {\|f\|_p} : f \in L^p(\calU^{n}), f\neq 0 \right\}.
\]
\end{thm}

This is motivated by recent work of Zhu \cite{Zhu06}, Dostani\'c \cite{Dos082}
and the author of the present paper \cite{Liu15}, in which sharp estimates for the norm of the Bergman projection over the unit ball
of $\bbC^n$ were obtained. It is also worth mentioning that, in the recent years, there has been increasing
interest in the study of the size of Bergman projection in various context other than the Bergman space.
See \cite{KM14,KV15,Per12,Per13,Per14,Vuj13}.

In the course of proving Theorem \ref{thm:main0}, we will precisely evaluate
the $L^p$ operator norm of a family of integral operators as follows.
For $\alpha>-1$, we define
\[
(T_{\alpha} f) (z) := \frac {(n+1)!}{4\pi^{n+1}} \int\limits_{\calU^n} \frac {\bfrho(w,w)^{\alpha}}
{\left|\bfrho(z,w)\right|^{n+2+\alpha}}
f(w) dV(w),
\]
%\[
%T_{\alpha} f (z) := \frac {(n+1)!}{4\pi^{n+1}} \int\limits_{\calU^n} \frac {\left(\ImPt w_{n+1} - |w^{\prime}|^2\right)^{\alpha}}
%{\left|\frac {i}{2} (\bar{w}_{n+1}-z_{n+1}) - z^{\prime} \cdot \overline{w^{\prime}}\right|^{n+2+\alpha}}
%f(w) dV(w),
%\]
for, say, continuous $f$ of compact support. It is a bounded map of $L^p(\calU^n)$ to itself, as \cite[Lemma 2.8]{CR80} shown.
Our second main result is the following.

\begin{thm}\label{thm:main2}
If $1\leq p<\infty$ and $p(1+\alpha)>1$ then
\begin{equation}\label{eqn:Talphanorm}
\|T_{\alpha}\|_{p\to p}=\frac {(n+1)! \Gamma\left(1+\alpha-\frac {1}{p}\right) \Gamma\left(\frac {1}{p}\right)}
{\Gamma^2\left(\frac {n+2+\alpha}{2}\right)}.
\end{equation}
\end{thm}

Note that $\|P_{\calU^{n}}\|_{p\to p} \leq \|T_0\|_{p\to p}$, and hence the second inequality in
\eqref{eqn:bergmanestimate} follows immediately from \eqref{eqn:Talphanorm}, together with the well-known formula
\[
\Gamma(x)\Gamma(1-x)=\frac {\pi}{\sin(\pi x)}.
\]
We also remark that when $n=0$, Theorem \ref{thm:main2} gives an affirmative answer to a conjecture of Dostanic in
\cite{Dos10} (a partial answer to this conjecture was given in \cite{LZ13}).

Recall that the Berezin transform over $\calU^n$ is defined by
\begin{align*}
(B_{\calU^n} f) (z) ~:=~& \int\limits_{\calU^n}\frac {|K_{\calU^n}(z,w)|^2} {K_{\calU^n}(z,z)} f(w) dV(w)\\
=~& \frac {(n+1)!}{4\pi^{n+1}} \int\limits_{\calU^n} \frac {\bfrho(z,z)^{n+2}}
{\left|\bfrho(z,w)\right|^{2(n+2)}}
f(w) dV(w),
\end{align*}
which plays important roles in Berezin's theory of quantization as well as in
the theory of Toeplitz operators.
Note that $B_{\calU^n}=T_{n+2}^{\ast}$, the adjoint of $T_{n+2}$. Another
immediate consequence of Theorem \ref{thm:main2} is the following.

\begin{cor}\label{thm:main1}
If $1<p\leq\infty$ then
\[
\|B_{\calU^n}\|_{p\to p} = \frac {\pi} {p\sin \frac {\pi}{p}} \prod_{k=1}^{n+1} \left(1+\frac 1{kp}\right).
\]
When $p=\infty$, the quantity on the right hand side
should be interpreted as $1$.
\end{cor}
This is an analogue of the main results in \cite{Dos08} and \cite{LZ12}.

Our third main result concerns the Cauchy-Szeg\"o projection over $b\calU^n$.
For $f$ holomorphic on $\calU^n$, we define
\begin{equation*}
\|f\|_{H^p} := \sup_{t>0} \Bigg\{ \int\limits_{b\calU^n} |f(u+t\bfi)|^p d\bfbeta(u) \Bigg\}^{\frac {1}{p}},
\end{equation*}
where $\bfi=(0^{\prime}, i)\in \bbC^{n+1}$ and the measure $d\bfbeta$ on $b\calU^n$ is defined by
the formula
\begin{equation*}
\int\limits_{b\calU^n} f d\bfbeta = \int\limits_{\bbC^n\times \bbR}
f(z^{\prime}, t+i|z^{\prime}|^2) dz^{\prime} dt,
\end{equation*}
for (say) continuous $f$ of compact support. See Section 2.1 below.
Then we set
\[
H^p(\calU^n) := \left\{ f: f \text{ is holomorphic on }\calU^n, \|f\|_{H^p} < \infty \right\},
\]
which is the analogue for $\calU^n$ of the classical Hardy space $H^p$ of holomorphic functions in the
upper half-plane.

The space $H^2(\calU^n)$ can be identified with the closed subspace of $L^2(b\calU^n)$
consisting of functions $\{ f^b\}$ that are boundary values of functions $f\in H^2(\calU^n)$,
so there exists an orthogonal projection from $L^2(b\calU^n)$ onto $H^2(\calU^n)$.
We denote this projection by $C_{\calU^n}$ and call it the Cauchy-Szeg\"o projection.
It may also be written as the Cauchy-Szeg\"o integral
\[
(C_{\calU^n} f) (z) = \int\limits_{b\calU^n} S_{\calU^n}(z,u) f(u) d\bfbeta(u), \qquad z\in \calU^n,
\]
where
\begin{equation}
S_{\calU^n}(z,w)= \frac {n!}{4\pi^{n+1}}\,\bfrho(z,w)^{-n-1}.
\end{equation}
See \cite[p.61, Proposition 5.3]{Gin64} or \cite[p.536, Theorem 1]{Ste93}. A classical theorem by Kor\'anyi and V\'agi
\cite[Theorem 6.1]{KV71} asserts that the Cauchy-Szeg\"o projection has an extension to a bounded operator from $L^p(b\calU^n)$ onto
$H^p(\calU^n)$, for $1<p<\infty$.

Our third main result gives a lower bound for the operator norm of $C_{\calU^n}$.

\begin{thm}\label{thm:normofCS}
For all $1<p<\infty$,
\begin{equation}\label{eqn:normofCS}
\|C_{\calU^{n}} \|_{L^p(b\calU^{n})\to H^p(\calU^n)} ~\geq~
\frac {\Gamma\left(\frac {n+1}{p}\right) \Gamma\left(\frac {n+1}{q}\right)}{
\Gamma^2 \left(\frac {n+1}{2}\right)},
\end{equation}
where $q:=\frac {p}{p-1}$ is the conjugate exponent of $p$.
\end{thm}

We shall deduce Theorem \ref{thm:normofCS} from Theorem \ref{thm:main0}, with the help of the following
inequality, which makes a connection between the norms of the two operators and might be of independent interest.

\begin{thm}\label{thm:twonorms}
For all $1<p<\infty$,
\begin{equation}\label{eqn:normcomparison}
\|C_{\calU^{n}} \|_{L^p(b\calU^{n})\to H^p(\calU^n)} ~\geq~ \|P_{\calU^{n-1}}\|_{L^p(\calU^{n-1})\to A^p(\calU^{n-1})}.
\end{equation}
\end{thm}

The above results suggest the following.
\begin{conj}\label{conj:ourconj}
For all $1<p<\infty$, we have
\begin{align}
\|P_{\calU^{n}} \|_{L^p(\calU^{n})\to A^p(\calU^n)} ~=~&
\frac {\Gamma\left(\frac {n+2}{p}\right) \Gamma\left(\frac {n+2}{q}\right)}{
\Gamma^2 \left(\frac {n+2}{2}\right)} \label{eqn:conj1}
\intertext{and}
\|C_{\calU^{n}} \|_{L^p(b\calU^{n})\to H^p(\calU^n)} ~=~&
\frac {\Gamma\left(\frac {n+1}{p}\right) \Gamma\left(\frac {n+1}{q}\right)}{
\Gamma^2 \left(\frac {n+1}{2}\right)}, \label{eqn:conj2}
\end{align}
where $q:=\frac {p}{p-1}$ is the conjugate exponent of $p$.
\end{conj}

Note that when $n=0$, the conjectured \eqref{eqn:conj2} reads
\begin{equation}\label{eqn:GKconj}
\|C_{\bbC_{+}} \|_{L^p(\bbR)\to H^p(\bbC_{+})} =
\frac {1}{\sin \frac {\pi}{p}}.
\end{equation}
This coincides with a variant of the Gohberg-Krupnik conjecture, which was proved
by Hollenbeck and Verbitsky \cite{HV00} only in 2000.
For the proof of \eqref{eqn:GKconj}, see \cite[p.373]{Kru10}.
This provides a support for our conjecture.
It is also noteworthy that the conjectured \eqref{eqn:conj2} would imply \eqref{eqn:conj1},
in view of Theorems \ref{thm:main0} and \ref{thm:twonorms}.
See also \cite{Liu15} and \cite{Liu16} for their counterparts in the setting of the unit ball.

The rest of the paper is organized as follows: In Section 2  we review some definitions and basic facts, and
in Section 3 we establish several technical lemmas, some of them might be of independent interest.
Sections 4 is devoted to the proof of Theorem \ref{thm:main2}.
Our first main result, Theorem \ref{thm:main0} will be proved in Sections 5.
Sections 6 is devoted to the proof of Theorem \ref{thm:twonorms}. 

\subsection*{Acknowledgements}

The author is indebted to Lifang Zhou for correcting two errors in an earlier version of the paper, and to
Guangbin Ren and Xieping Wang for many helpful comments.

\section{Preliminaries}

\subsection{Heisenberg group}

We recall the definition of the Heisenberg group and some basic facts which can be found in \cite[Chapter XII]{Ste93}
and \cite[Chapter 9]{Kra09}.

We denote by $\bbH^n$ the Heisenberg group, that is, the set
\[
\bbC^n \times \bbR = \{ [\zeta,t]: \zeta\in \bbC^n, t\in \bbR\}
\]
endowed with the Heisenberg group operation
\[
[\zeta,t]\cdot [\eta,s]=[\zeta+\eta, t+s+2 \mathrm{Im}(\zeta\cdot \bar{\eta})].
\]
Here we shall use square brackets [\, ] for elements of the Heisenberg group to distinguish them from
points in $\bbC^{n+1}$, for which parentheses (\,) are used. The identity element is $[0,0]$ and the inverse of
$[\zeta,t]$ is $[\zeta,t]^{-1}=[-\zeta,-t]$.
The Haar measure $dh$ on $\bbH^n$ is the usual Lebesgue measure $d\zeta dt$ on $\bbC^n\times \bbR$, here
we write $h=[\zeta,t],\, \zeta\in \bbC^n, t\in \bbR$. (To be more precise, $d\zeta = d\eta d\xi$ if $\zeta = \eta + i\xi$ with
$\eta,\xi\in \bbR^n$.)

To each element $h=[\zeta,t]$ of $\bbH^n$, we associate the following (holomorphic) affine self-mapping of
$\calU^n$:
\begin{equation}\label{eqn:groupaction}
h:\; (z^{\prime},z_{n+1}) \longmapsto (z^{\prime}+\zeta, z_{n+1}+t+ 2i z^{\prime}\cdot \bar{\zeta} + i|\zeta|^2).
\end{equation}
The mappings \eqref{eqn:groupaction} are simply transitive on the boundary $b\calU^n$ of $\calU^n$,
so we can identify the Heisenberg group with $b\calU^n$ via its action on the origin
\[
\bbH^n \ni [\zeta,t] \,\longmapsto\, (\zeta, t+i|\zeta|^2) \in b\calU^n.
\]
This identification allows us to transport the Haar measure $dh$ on $\bbH^n$ to a measure $d\bfbeta$ on $b\calU^n$;
that is, we have the integration formula
\begin{equation}\label{eqn:defnofbeta}
\int\limits_{b\calU^n} f d\bfbeta = \int\limits_{\bbC^n\times \bbR}
f(z^{\prime}, t+i|z^{\prime}|^2) dz^{\prime} dt,
\end{equation}
for (say) continuous $f$ of compact support. The measure $d\bfbeta$ is invariant under the action of $\bbH^n$,
that is, $d\bfbeta(h(z))=d\bfbeta(z)$ for each $h\in \bbH^n$.

We shall make frequent use of the following Fubini's theorem on $\calU^n$:
\begin{equation}\label{eqn:fubini}
\int\limits_{\calU^n} f dV = \int\limits_{0}^{\infty} \int\limits_{b\calU^n} f(u+t\bfi) d\bfbeta(u) dt,
\end{equation}
which is valid for, say, continuous $f$ of compact support. This can be easily verified by substituting
\eqref{eqn:defnofbeta} into the right hand side of \eqref{eqn:fubini}

Finally, it is easy to verify that
\begin{equation}\label{eqn:h-invariance4rho}
\bfrho(h(z),h(w))=\bfrho(z,w),
\end{equation}
for each $h\in \bbH^n$, and for all $z\in \calU^n$, $w\in b\calU^n$.

\subsection{Cayley transform}
%$\calU^n$ is the image of the complex unit ball $\bbB^{n+1}$ under the Cayley transformation $\Phi$ defined by
%\[
%\Phi:\ (z^{\prime}, z_{n+1})\; \longmapsto\; \left( \frac {z^{\prime}}{1+z_{n+1}},
%i\frac {1-z_{n+1}}{1+z_{n+1}} \right).
%\]
%Indeed, $\Phi$ maps $\bbB^{n+1}$ to $\calU^n$ in a holomorphic, one-to-one, and onto fashion, and
%hence is a biholomorphic isomorphism.
Recall that the Cayley transform $\Phi:\bbB_{n+1}\to \calU^n$ is given by
\[
(z^{\prime}, z_{n+1})\; \longmapsto\; \left( \frac {z^{\prime}}{1+z_{n+1}},
i\left(\frac {1-z_{n+1}}{1+z_{n+1}}\right) \right).
\]
It is immediate to calculate that
\[
\Psi:=\Phi^{-1}: \left(z^{\prime},z_{n+1}\right)\;\longmapsto\; \left(\frac {2iz^{\prime}}{i+z_{n+1}},
\frac {i-z_{n+1}}{i+z_{n+1}}\right).
\]

Again, we refer to \cite[Chapter XII]{Ste93} and \cite[Chapter 9]{Kra09} for the properties
of these two mappings. For the convenience of later reference, we record the following lemma.
\begin{lem}
The mappings $\Phi$ and $\Psi$ have the following elementary properties:
\begin{enumerate}
\item[(i)]
The identity
\begin{equation}\label{eqn:identity14phi}
\bfrho(\Phi(\xi),\Phi(\eta)) = \frac {1-\xi\cdot \overline{\eta}} {(1+\xi_{n+1}) (1+\overline{\eta}_{n+1})}
\end{equation}
holds for all $\xi,\eta\in \ball$.
\item[(ii)]
The real Jacobian of $\Phi$ at $\xi\in \ball$ is
\begin{equation}\label{eqn:jacobian4phi}
\left(J_{R}\Phi\right)(\xi) = \frac {4}{|1+\xi_{n+1}|^{2(n+2)}}.
\end{equation}
\item[(iii)]
The identity
\begin{equation}\label{eqn:cayleyidentity1}
1- \Psi(z)\cdot \overline{\Psi(w)} = \frac {\bfrho(z,w)} {\bfrho(z,\bfi)\bfrho(\bfi,w)}
\end{equation}
holds for all $z,w\in \calU^{n}$, where $\bfi=(0^{\prime},i)$.
\item[(iv)]
The identity
\begin{equation}\label{eqn:cayleyidentity2}
|\Psi(z)|^2 = 1- \frac {\bfrho(z,z)} {|\bfrho(z,\bfi)|^2}
\end{equation}
holds for all $z\in \calU^{n}$.
\item[(v)]
The real Jacobian of $\Psi$ at $z\in \calU^n$ is
\begin{equation}\label{eqn:jacobian}
\left(J_{R}\Psi\right)(z) = \frac {1}{4|\bfrho(z,\bfi)|^{2(n+2)}}.
\end{equation}
%\item[(iv)]
%\begin{equation}
%\bfrho(z,w)=\frac {1- \Psi(z)\cdot\overline{\Psi(w)} } {(1-[\Psi(z)]_{n+1})
%(1- \overline{[\Psi(w)]}_{n+1})}
%\end{equation}
%holds for all $z\in \calU^{n}$ and $w\in b\calU^n$.
%\item[(v)]
%\begin{equation}
%d\sigma(\Psi(z))=\frac {n!}{4\pi^{n+1}} |\bfrho(z,\bfi)|^{-2(n+1)} d\bfbeta(z).
%\end{equation}
%\item[(vi)]
%\begin{equation}
%d\bfbeta(\Phi(\xi))=\frac {4\pi^{n+1}} {n!} |1-\xi_{n+1}|^{-2(n+1)} d\sigma(\xi).
%\end{equation}
\end{enumerate}

\end{lem}

%If $z=\Phi(\eta)$ and $w=\Phi(\xi)$, then
%\begin{equation}
%\bfrho(z,w)=\frac {1-\eta \bar{\xi}} {(1-\eta_1)(1-\bar{\xi}_1)}.
%\end{equation}
%and
%\begin{equation}
%1-\eta \bar{\xi} = \frac {\bfrho(z,w)} {\bfrho(z,\bfi)\bfrho(\bfi,w)}.
%\end{equation}
%where $\bfi=(i,0^{\prime})$.

%\[
%\left(J_{R}\Phi\right)(\eta)=\frac {4}{|1-\eta_1|^{2(n+1)}}.
%\]
Note that the mappings $\Phi$ and $\Psi$ extend also to the boundaries of the domains $\ball$ and $\calU^n$.
Thus, $b\calU^n$ corresponds via $\Psi$ to the unit sphere $\sphere$, except for the ``south pole''
$(0^{\prime},-1)$.
It is easy to check that the identity
\begin{equation}\label{eqn:rhopsi}
\bfrho(z,w)=\frac {1- \Psi(z)\cdot\overline{\Psi(w)} } {(1+[\Psi(z)]_{n+1})
(1+ \overline{[\Psi(w)]}_{n+1})}
\end{equation}
holds for all $z\in \calU^{n}$ and $w\in b\calU^n$.

Finally, writing $d\sigma$ for the normalized surface measure on the unit sphere $\sphere$, one has
the following change of variables formula (see \cite[p.575, 7.2(b)]{Ste93}):
\begin{equation}\label{eqn:changeofv}
\int\limits_{b\calU^n} f d\bfbeta ~=~ \frac {4\pi^{n+1}} {n!} \int\limits_{\sphere}
\frac {f(\Phi(\xi))} {|1+\xi_{n+1}|^{2(n+1)}} d\sigma(\xi).
\end{equation}

\subsection{M\"obius transformations}

The group of all one-to-one holomorphic mappings of $\ball$ onto $\ball$ (the so-called automorphisms
of $\ball$) will be denoted by $\mathrm{Aut}(\ball)$. It is generated by the unitary transformations
on $\bbC^{n+1}$ along with the M\"obius transformations
$\varphi_{\xi}$ given by
\[
\varphi_{\xi}(\eta) := \frac {\xi-P_{\xi}\eta-(1-|\xi|^2)^{\frac {1}{2}}Q_{\xi}\eta}{1-\eta\cdot\overline{\xi}},
\]
where $\xi\in \ball$, $P_{\xi}$ is the orthogonal projection onto the space spanned
by $\xi$,
%, which is explicitly given by $X_0=0$, and
%\[
%X_{\xi}\eta := \frac {\eta\cdot \overline{\xi}} {|\xi|^2} \xi, \quad \text{if }\xi\neq 0,
%\]
and $Q_{\xi}\eta=\eta- P_{\xi}\eta$.

It is easily shown that the mapping $\varphi_{\xi}$ satisfies
\[
\varphi_{\xi}(0)=\xi, \quad \varphi_{\xi}(\xi)=0, \quad \varphi_{\xi}(\varphi_{\xi}(\eta))=\eta.
\]
Furthermore, for all $\eta, \omega\in \overline{\ball}$,
\begin{align}
1- \varphi_{\xi}(\eta)\cdot \overline{\varphi_{\xi}(\omega)} ~=~& \frac {(1- |\xi|^2)
(1- \eta\cdot \overline{\omega})} {(1-\eta\cdot\overline{\xi})(1- \xi\cdot \overline{\omega})}, \label{eqn:moeb0}
\intertext{and in particular,}
1- \varphi_{\xi}(\eta) \cdot \overline{\xi} ~=~& \frac {1-|\xi|^2} {1- \eta \cdot \overline{\xi}}.\label{eqn:moeb1}
\end{align}
Finally, an easy computation shows that
\begin{equation}\label{eqn:moeb2}
1- \varphi_{\xi}(\eta)\cdot \overline{\omega} ~=~ \frac {(1- \eta \cdot \overline{\varphi_{\xi}(\omega)})
(1- \xi\cdot \overline{\omega})} {1- \eta\cdot \overline{\xi}}
\end{equation}
holds for all $\eta,\omega\in \overline{\ball}$.

The best general reference here is \cite[Chapter 2]{Rud80}.

\subsection{Hypergeometric functions}

We use the classical notation %$\hyperg{a}{b}{c}{\lambda}$ to denote
\begin{equation*}\label{eq:hypergdefin}
\hyperg{a}{b}{c}{\lambda}:=\sum_{k=0}^{\infty}\frac{(a)_k(b)_k}{(c)_k}\frac{\lambda^k}{k!}
\end{equation*}
with $c\neq 0, -1,-2,\ldots$, where
\[
(a)_0=1,\quad(a)_k=a(a+1)\ldots(a+k-1)
\quad \text{ for } k\geq1.
\]
denotes the Pochhammer symbol of $a$. This series gives an analytic function for $|\lambda|<1$,
called the Gauss hypergeometric function associated to $(a,b,c)$.

We refer to \cite[Chapter 2]{AAR99} for the properties of these
functions. Here, we only record three formulas for later reference.
\begin{align}
\hyperg{a}{b}{c}{1} ~=~& \frac {\Gamma(c) \Gamma(c-a-b)}
{\Gamma(c-a) \Gamma(c-b)},\qquad \RePt(c-a-b)>0.
\label{eqn:gauss}\\
\hyperg{a}{b}{c}{\lambda} ~=~& (1-\lambda)^{c-a-b} \hyperg{c-a}{c-b}{c}{\lambda}. \label{eqn:euler}\\
\hyperg{a}{b}{c}{\lambda} ~=~& \frac{\Gamma(c)}{\Gamma(b)\Gamma(c-b)}\int_0^1
t^{b-1}(1-t)^{c-b-1} (1-t\lambda)^{-a} dt,\notag\\
&\hspace{64pt} \RePt c>\RePt b>0;\;  |\arg(1-\lambda)|<\pi. \label{eqn:euler2}
\end{align}

%\begin{lem}
%If $a>-1$ then
%\[
%\left|(1-t\lambda)^{-a} - (1-t)^{-a} \right| \leq
%|a|\cdot (1-t)^{-a-1} |1-\lambda|,
%\]
%otherwise,
%\[
%\left|(1-t\lambda)^{-a} - (1-t)^{-a} \right| \leq  |a|\cdot 2^{-a-1} |1-\lambda|.
%\]
%\end{lem}

%\begin{lem}
%If $a>-1$ and $c>\max\{a+b+1,a+1\}$, then
%\begin{equation}
%\left|\hyperg {a}{b}{c}{\lambda} - \hyperg {a}{b}{c}{1}\right| ~\leq~
%\frac {|a|\, \Gamma(c)\Gamma(c-a-b-1)} {\Gamma(c-a-1)\Gamma(c-b)}\, |1-\lambda|
%\end{equation}
%for all $\lambda\in \disk$; If $a<-1$ and $c>a+b$, then
%\begin{equation}
%\left|\hyperg {a}{b}{c}{\lambda} - \hyperg {a}{b}{c}{1}\right| ~\leq~
%|a|\cdot 2^{-a-1} \, |1-\lambda|.
%\end{equation}
%for all $\lambda\in \disk$.
%\end{lem}

\subsection{Schur's test}
The following lemma, usually called Schur's test, is one of the most commonly used
results for proving the $L^p$-boundedness of integral
operators. See, for example, \cite{Zhu07}.

\begin{lem}\label{lem:schurtest}
Suppose that $(X,\mu)$ is a $\sigma$-finite measure space and
$Q(x,y)$ is a nonnegative measurable function on $X\times X$ and $T$
is the associated integral operator
\[
Tf(x)=\int_X Q(x,y) f(y) d\mu(y).
\]
Let $1<p<\infty$ and $\frac {1}{p}+\frac {1}{q}=1$. If there exist a positive constant
$C$ and a positive measurable function $g$ on $X$ such that
\[
\int_X Q(x,y) g(y)^{q} d\mu(y) \leq  C g(x)^{q}
\]
for almost every $x$ in $X$ and
\[
\int_X Q(x,y) g(x)^{p} d\mu(x) \leq  C g(y)^{p}
\]
for almost every $y$ in $X$, then $T$ is bounded on $L^p(X,\mu)$ with
$\|T\|\leq C$.
\end{lem}

\section{Technical lemmas}

\subsection{An elementary inequality}

\begin{lem}\label{lem:est4hyperg}
Let $a, b\in \bbR$ and $c>\max\{a+b+1,a+1,b\}$. Then
\begin{equation}\label{eqn:est4hyperg1}
\left|\hyperg {a}{b}{c}{\lambda} - \hyperg {a}{b}{c}{1}\right| ~\leq~
C(a,b,c)\cdot |1-\lambda|
\end{equation}
for all $\lambda\in \disk$, where
\[
C(a,b,c):=|a| \cdot \max\left\{2^{-a-1}, \frac {\Gamma(c)\Gamma(c-a-b-1)} {\Gamma(c-a-1)\Gamma(c-b)} \right\}.
\]
\end{lem}

\begin{proof}
Since $c>b$ and $c-a-b>0$, by \eqref{eqn:euler2}, we have
\begin{align}
& \left|\hyperg {a}{b}{c}{\lambda} - \hyperg {a}{b}{c}{1}\right| \notag \\
& \qquad ~\leq~ \frac{\Gamma(c)}{\Gamma(b)\Gamma(c-b)} \int_0^1
t^{b-1}(1-t)^{c-b-1} \left|(1-t\lambda)^{-a}-(1-t)^{-a}\right| dt \label{eqn:est4hyperg2}
\end{align}
for all $\lambda\in \disk$. By the mean value theorem we have
\begin{align*}
\left|(1-t\lambda)^{-a} - (1-t)^{-a} \right|
&~\leq~  \sup_{\vartheta \in t\disk} \left|a\, (1-\vartheta)^{-a-1}\right|\cdot |1-\lambda|\\
&~\leq~  |a|\cdot \max \left\{ 2^{-a-1}, (1-t)^{-a-1}\right\} |1-\lambda|.
\end{align*}
Substituting this into \eqref{eqn:est4hyperg2} yields the desired inequality \eqref{eqn:est4hyperg1}.
\end{proof}

\subsection{A Forelli-type formula}

The following lemma deals with integration on $b\calU^n$ of functions of fewer variables,
which is in the same spirit as a result of Forelli \cite[p.383]{For74} (see also \cite[p.14]{Rud80} or \cite[p.10, Lemma 1.9]{Zhu05}).

Suppose $0\leq k \leq n$ and let $\calU^k$ be the Siegel upper half-space in $\bbC^{k+1}$. It is convenient to let
\[
\bfrho_k (w) ~:=~ \ImPt w_{k+1} - \sum_{j=1}^k |w_j|^2
\]
We think of $\bfrho_k$ as a ``height function'' in $\calU^k$. Note that
\[
\bfrho_n(z) ~=~ \bfrho(z,z) ~=~ \ImPt z_{n+1} - |z^{\prime}|^2.
\]

\begin{lem} \label{lem:prjformula}
Suppose $0\leq k < n$ and $f$ is a function on $b\calU^n$ that depends only on $z_{n-k+1},\cdots,z_{n+1}$.
Then $f$ can be regarded as defined on $\calU^{k}$ and
\begin{equation*}\label{eqn:prjformula}
\int\limits_{b\calU^n} (f\circ \Pi_k) d\beta = \frac {\pi^{n-k}} {(n-k-1)!}  \int\limits_{\calU^{k}}
\bfrho_k(w)^{n-k-1} f(w) dm_{2k+2}(w),
\end{equation*}
where $\Pi_k$ is the orthogonal projection of $\bbC^{n+1}$ onto $\bbC^{k+1}$ given by
\[
(z_1,\cdots,z_{n+1}) ~\longmapsto~ (z_{n-k+1},\cdots,z_{n+1}),
\]
and
\[
c_{n,k}:= \frac {\pi^{n-k}} {(n-k-1)!}.
\]
In particular when $k=n-1$, this reads
\begin{equation}\label{eqn:prjformula2}
\int\limits_{b\calU^n} (f\circ \Pi_{n-1}) d\beta = \pi  \int\limits_{\calU^{n-1}}
 f(w)  dm_{2n}(w).
\end{equation}
\end{lem}

\begin{proof}
For convenience, we use the notation $z=(z^{\dag}, z^{\ddag})$,  where
\[
z^{\dag}=(z_1,\ldots,z_{n-k})\in \bbC^{n-k}
\quad \text{and} \quad z^{\ddag} =(z_{n-k+1},\cdots,z_{n+1}) \in \bbC^{k+1}.
\]
By an approximation argument, it suffices for us to prove the result when $f$ is continuous
in $\bbC^{k+1}$ and has support in $\left\{z^{\ddag} \in \bbC^{k+1}: \bfrho_k (z^{\ddag})>r_0\right\}$
for some $r_0> 0$.

Fix such an $f$ and consider the integrals
\[
I(r)=\int\limits_{\{\bfrho_n (z)>r\}} (f\circ \Pi_k) dV, \qquad 0<r<\infty.
\]
By Fubini's theorem \eqref{eqn:fubini}, we have
\[
I(r)=\int\limits_{r}^{\infty} \Bigg\{\int\limits_{b\calU^n} (f\circ \Pi_k) (u+t\bfi) d\bfbeta(u) \Bigg\}dt.
\]
We then differentiate this to obtain
\begin{equation}\label{eqn:Iprime1}
I^{\prime}(0)= - \int\limits_{b\calU^n} (f\circ \Pi_k) (u) d\bfbeta(u).
\end{equation}
On the other hand, if $0<r<r_0$, an application of the classical Fubini's theorem shows that
\begin{align*}
I(r)~=~& \int\limits_{\{\bfrho_k(z^{\ddag})>r\}} \Bigg \{\int\limits_{\{|z^{\dag}|^2< \bfrho_k(z^{\ddag})-r\}}
(f\circ \Pi_k) (z^{\dag},z^{\ddag})  dm_{2n-2k}(z^{\dag})\Bigg\} dm_{2k+2}(z^{\ddag})\\
=~&\int\limits_{\{\bfrho_k(z^{\ddag})>r\}} \frac {\pi^{n-k}} {(n-k)!} \left(\bfrho_k(z^{\ddag})-r\right)^{n-k}
f(z^{\ddag})  dm_{2k+2}(z^{\ddag})\\
=~& \frac {\pi^{n-k}} {(n-k)!}  \int\limits_{\calU^{k}} \left(\bfrho_k(z^{\ddag})-r\right)^{n-k}
f(z^{\ddag})  dm_{2k+2}(z^{\ddag}),
\end{align*}
where the last equality follows from the assumption that $f$ is supported in
$\{z^{\ddag}\in \bbC^{k+1}: \bfrho_k(z^{\ddag})>r_0\}$.
Differentiation then gives
\begin{equation}\label{eqn:Iprime2}
I^{\prime}(0) = - \frac {\pi^{n-k}} {(n-k-1)!}  \int\limits_{\calU^{k}} \bfrho_k(z^{\ddag})^{n-k-1}
f(z^{\ddag})  dm_{2k+2}(z^{\ddag}).
\end{equation}
Comparison of \eqref{eqn:Iprime1} and \eqref{eqn:Iprime2} gives \eqref{eqn:prjformula}.
\end{proof}

\begin{cor}
Let $f$ be a function of one complex variable. Then, for any $z\in b\calU^n$, we have
\begin{equation}
\int\limits_{b\calU^n} f(2i\bfrho(w,z)) d\bfbeta(w) ~=~
\frac {\pi^n}{(n-1)!} \int\limits_{\bbC_{+}} f(\lambda) (\ImPt \lambda)^{n-1} dm_{2}(\lambda).
\end{equation}
\end{cor}

\begin{proof}
Let $z\in b\calU^n$ be fixed. We put $h:=[-z^{\prime}, -z_{n+1}+i|z^{\prime}|^2]\in \bbH^n$.
Then $h(z)=0$. Hence, by \eqref{eqn:h-invariance4rho}, we see that
\begin{align*}
\int\limits_{b\calU^n} f(2i\bfrho(w,z)) d\bfbeta(w) ~=~& \int\limits_{b\calU^n} f(2i\bfrho(h(w),0)) d\bfbeta(w)\\
=~& \int\limits_{b\calU^n} f(2i\bfrho(u,0)) d\bfbeta(u) ~=~ \int\limits_{b\calU^n} f(u_{n+1}) d\bfbeta(u),
\end{align*}
where the second equality follows from the $\bbH^n$-invariance of $d\bfbeta$.
Finally, an application of Lemma \ref{lem:prjformula} (with $k=0$) completes the proof.
\end{proof}

\begin{cor}
Suppose $1\leq k < n$. Then for $f\in L^1(b\calU^n)$ we have
\begin{equation*}
\int\limits_{b\calU^n} f d\bfbeta = c_{n,k} \int\limits_{\calU^{k}} \Bigg\{
\int\limits_{\mathbb{S}_{n-k}} f(\sqrt{\bfrho_k(w)}\, \eta, w) d\sigma_{n-k}(\eta) \Bigg\} \bfrho_k(w)^{n-k-1} dm_{2k+2}(w).
\end{equation*}
In particular, when $k=n-1$,
\begin{equation}\label{eqn:forellispcase1}
\int\limits_{b\calU^n} f d\bfbeta = \pi \int\limits_{\calU^{n-1}} \Bigg\{ \frac {1}{2\pi}
\int\limits_{0}^{2\pi} f(\sqrt{\bfrho_{n-1}(w)}\, e^{i\theta}, w) d\theta \Bigg\}  dm_{2n}(w).
\end{equation}
\end{cor}

\begin{proof}
As in the proof of Lemma \ref{lem:prjformula}, we write $z=(z^{\dag}, z^{\ddag})$, where
$z^{\dag}\in \bbC^{n-k}$ and $z^{\ddag}\in \bbC^{k+1}$. Then
\[
\int\limits_{b\calU^n} f d\bfbeta = \int\limits_{b\calU^n} f(z^{\dag}, z^{\ddag}) d\bfbeta (z).
\]
Note that if $z=(z^{\dag}, z^{\ddag})\in b\calU^n$ then $|z^{\dag}|^2 = \bfrho_k(z^{\ddag})$.
Since $d\bfbeta(\mathbf{g}(z))=d\bfbeta(z)$ for every unitary linear transformation $\mathbf{g}$ on $\bbC^n$,
where $\mathbf{g}((z^{\prime}, z_{n+1})) := (\mathbf{g}(z^{\prime}),z_{n+1})$,
\[
\int\limits_{b\calU^n} f(z^{\dag}, z^{\ddag}) d\bfbeta (z) = \int\limits_{b\calU^n} f(|z^{\dag}|\,\eta, z^{\ddag}) d\bfbeta (z)
= \int\limits_{b\calU^n} f(\sqrt{\bfrho_k(z^{\ddag})}\, \eta, z^{\ddag}) d\bfbeta (z)
\]
for any fixed $\eta\in \mathbb{S}_{n-k}$. Integrating over $\eta\in \mathbb{S}_{n-k}$ and applying Fubini's
theorem, we obtain
\[
\int\limits_{b\calU^n} f d\bfbeta = \int\limits_{b\calU^n} \Bigg\{ \int\limits_{\mathbb{S}_{n-k}}
 f(\sqrt{\bfrho_k(z^{\ddag})}\, \eta, z^{\ddag}) d\sigma_{n-k}(\eta) \Bigg\} d\bfbeta (z).
\]
The inner integral above defines a function that only depends on the last $k+1$ variables.
Therefore, an application of Lemma \ref{lem:prjformula} completes the proof.
\end{proof}

\subsection{Evaluation of some integrals}

\begin{lem}\label{lem:keylemma}
Let $\theta>0$ and $\gamma>-1$. The identities
\begin{equation}\label{eqn:keylem1}
\int\limits_{b\calU^n} \frac {d\bfbeta(u)} {|\bfrho(z,u)|^{n+1+\theta}} ~=~
\frac {4\pi^{n+1} \Gamma(\theta)} {\Gamma^2\left(\frac {n+1+\theta}{2}\right)}\ \bfrho(z,z)^{-\theta}
\end{equation}
and
\begin{equation}\label{eqn:keylem2}
\int\limits_{\calU^n} \frac {\bfrho(w,w)^{\gamma} dV(w)} {|\bfrho(z,w)|^{n+2+\theta+\gamma}}  ~=~
\frac {4 \pi^{n+1} \Gamma(1+\gamma) \Gamma(\theta)} {\Gamma^2\left(\frac {n+2+\theta+\gamma}{2}\right)}\ \bfrho(z,z)^{-\theta}
\end{equation}
hold for all $z\in \calU^n$.
\end{lem}

\begin{proof}
For fixed $z\in \calU^n$, we put $h=[-z^{\prime}, -\mathrm{Re} z_{n+1}]\in \bbH^n$.
Recall that $h\in \bbH^n$ acts on $z\in \calU^n$ by \eqref{eqn:groupaction}. It is easy to check that $h(z) =(0,i\bfrho(z,z))$.
Note also that an element of $b\calU^n$ has the form $(w^{\prime},t+i|w^{\prime}|^2)$ with $t=\mathrm{Re} w_{n+1}$.
It follows that
\[
\bfrho(h(z),w)=\frac {|w^{\prime}|^2+\bfrho(z,z)+it}{2}
\]
for every $w\in b\calU^n$.
Since $d\bfbeta$ is $\bbH^n$-invariant, by making the change of variables $w\mapsto h^{-1}(w)$ and using \eqref{eqn:h-invariance4rho},
we obtain
\begin{align*}
\int\limits_{b\calU^n} \frac {d\bfbeta(w)} {|\bfrho(z,w)|^{n+1+\theta}} %~=~&
%\int\limits_{b\calU^n} \frac {d\bfbeta\left(h^{-1}(w)\right)} {|\bfrho(z,h^{-1}(w))|^{n+1+\theta}}
~=~&\int\limits_{b\calU^n} \frac {d\bfbeta(w)} {|\bfrho(h(z),w)|^{n+1+\theta}}\\
=~& 2^{n+1+\theta} \int\limits_{\bbC^n\times \bbR} \frac {dw^{\prime} dt}
{\left[(|w^{\prime}|^2+ \bfrho(z,z))^2+t^2\right]^{\frac {n+1+\theta}{2}}}.
\end{align*}
A simple scaling argument (which involves carrying out the $t$-integration first) shows that
the last integral equals
\begin{align*}
& \Bigg\{\int\limits_{-\infty}^{\infty} \frac {dt}{(1+t^2)^{\frac {n+1+\theta}{2}}} \Bigg\}
\Bigg\{\int\limits_{0}^{\infty} \frac {r^{2n-1}}{(1+r^2)^{n+\theta}} dr \Bigg\} \frac {2\pi^n}{\Gamma(n)}\, \bfrho(z,z)^{-\theta}\\
&\qquad =~ \frac {\Gamma\left(\frac {1}{2}\right)\Gamma\left(\frac {n+\theta}{2}\right)}{\Gamma(\frac {n+1+\theta}{2})} \frac {\Gamma(n)
\Gamma(\theta)} {2 \Gamma(n+\theta)} \frac {2\pi^n}{\Gamma(n)}\, \bfrho(z,z)^{-\theta}\\
&\qquad =~ 2^{1-n-\theta} \frac {\pi^{n+1} \Gamma(\theta)} {\Gamma^2(\frac {n+1+\theta}{2})}\, \bfrho(z,z)^{-\theta}.
\end{align*}
Here, in the first equality we have used the well-known identity
\[
\int\limits_{0}^{\infty} \frac {t^{2b-1}} {(1+t^2)^{a}} dt = \frac {\Gamma(b)\Gamma(a-b)} {2\Gamma(a)}, \qquad
\text{if } a>b,
\]
and, in the second equality we have used the ``duplication formula''
\[
\Gamma\left(\frac {1}{2}\right)\Gamma(2a) = 2^{2a-1} \Gamma(a)\Gamma\left(a+\frac {1}{2}\right).
\]
This proves \eqref{eqn:keylem1}.

We proceed to prove \eqref{eqn:keylem2}. By \eqref{eqn:fubini}, we have
\[
\int\limits_{\calU^n} \frac {\bfrho(u,u)^{\gamma} du} {|\bfrho(z,u)|^{n+2+\theta+\gamma}}
~=~ \int\limits_0^{\infty} \bigg\{ \int\limits_{b\calU^n}
\frac {d\bfbeta(w)} {|\bfrho(z, w+t\bfi)|^{n+2+\theta+\gamma}}\bigg\} t^{\gamma}dt.
\]
Note that $\bfrho(z, w+t\bfi)=\bfrho(z+t\bfi,w)$ and $\bfrho(z+t\bfi,z+t\bfi)=\bfrho(z,z)+t$.
Applying \eqref{eqn:keylem1} to the inner integral yields
\begin{align*}
\int\limits_{\calU^n} \frac {\bfrho(u,u)^{\gamma} du} {|\bfrho(z,u)|^{n+2+\theta+\gamma}}
~=~& \frac {4\pi^{n+1} \Gamma(1+\theta+\gamma)} {\Gamma^2\left(\frac {n+2+\theta+\gamma}{2}\right)}
\int\limits_0^{\infty} \frac {t^{\gamma} dt} {(\bfrho(z,z)+t)^{1+\theta+\gamma}}\\
~=~& \frac {4\pi^{n+1} \Gamma(1+\theta+\gamma)} {\Gamma^2\left(\frac {n+2+\theta+\gamma}{2}\right)}
\frac {\Gamma(1+\gamma)\Gamma(\theta)} {\Gamma(1+\theta+\gamma)} \bfrho(z,z)^{-\theta},
\end{align*}
as desired.
\end{proof}

\begin{lem}\label{lem:crucial2}
\begin{enumerate}
\item[(i)]
If $a\in \mathbb{R}$ and $\max\{b,c, b+c\}<n+1$, then
\begin{align}\label{eqn:crucial3}
&\int\limits_{\sphere} \frac {d\sigma(\omega)} {(1-\eta\cdot \overline{\omega})^{a} (1-\zeta\cdot \overline{\omega})^{b}
(1-\omega\cdot \overline{\zeta})^{c}} \notag\\
&\qquad ~=~ \frac {n! \Gamma(n+1-b-c)} {\Gamma(n+1-b)\Gamma(n+1-c)} \hyperg {a}{c}{n+1-b} {\eta\cdot \overline{\zeta}}
\end{align}
holds for any $\eta\in \ball$ and $\zeta\in \sphere$.
\item[(ii)]
If $a\in \mathbb{R}$ and $\max\{b,c, b+c\}<n+2$, then
\begin{align}\label{eqn:crucial4}
&\int\limits_{\ball} \frac {dV(\xi)}
{(1-\eta\cdot \overline{\xi})^{a} (1-\zeta\cdot \overline{\xi})^{b}
(1-\xi\cdot \overline{\zeta})^{c}} \notag\\
&\qquad ~=~ \frac {\pi^{n+1} \Gamma(n+2-b-c)} {\Gamma(n+2-b)\Gamma(n+2-c)} \hyperg {a}{c}{n+2-b} {\eta\cdot \overline{\zeta}}
\end{align}
holds for any $\eta\in \ball$ and $\zeta\in \sphere$.
\end{enumerate}
\end{lem}

\begin{proof}

We may further assume that $b+c<0$; if we
prove the lemma in this special case, the general case
follows by analytic continuation.
We prove only the first part of the lemma, the proof of the second part being similar.

According to \cite[Lemma 2.3]{Liu15}, the identity
\begin{align*}
&\int\limits_{\sphere} \frac {d\sigma(\omega)} {(1-\eta\cdot \overline{\omega})^{a} (1-r\zeta\cdot \overline{\omega})^{b}
(1-r\omega\cdot \overline{\zeta})^{c}} \notag\\
&\qquad ~=~ \sum_{j=0}^{\infty} \frac {(a)_{j} (c)_{j}} {(n+1)_{j} j!} \hyperg {b}{c+j}{n+1+j}{r^2} (r\eta\cdot \overline{\zeta})^j
\end{align*}
holds for all $r\in [0,1)$, $\eta\in \ball$ and $\zeta\in \sphere$.
Note that
\[
\left| \frac {1} {(1-\eta\cdot \overline{\omega})^{a} (1-r\zeta\cdot \overline{\omega})^{b}
(1-r\omega\cdot \overline{\zeta})^{c}}\right|
~\leq~ \frac {2^{-b-c}} {|1- \eta\cdot \overline{\omega}|^{a}},
\]
since $b+c<0$. Letting $r\to 1$, applying the dominated convergence theorem and using \eqref{eqn:gauss}, we obtain
\begin{align*}
&\int\limits_{\sphere} \frac {d\sigma(\omega)} {(1-\eta\cdot \overline{\omega})^{a}
(1-\zeta\cdot \overline{\omega})^{b} (1-\omega\cdot \overline{\zeta})^{c}} \notag\\
&\qquad ~=~ \sum_{j=0}^{\infty} \frac {(a)_{j} (c)_{j}} {(n+1)_{j} j!} \hyperg {b}{c+j}{n+1+j}{1}
(\eta\cdot \overline{\zeta})^j\\
&\qquad ~=~ \sum_{j=0}^{\infty} \frac {(a)_{j} (c)_{j}} {(n+1)_{j} j!} \frac {\Gamma(n+1+j) \Gamma(n+1-b-c)}
{\Gamma(n+1-b+j) \Gamma(n+1-c)} (\eta\cdot \overline{\zeta})^j\\
&\qquad ~=~ \frac {n! \Gamma(n+1-b-c)} {\Gamma(n+1-b)\Gamma(n+1-c)} \hyperg {a}{c}{n+1-b} {\eta\cdot \overline{\zeta}}
\end{align*}
as desired.
\end{proof}

\begin{lem}
For $\theta\in \bbR$ and $\gamma>-1$, the identity
\begin{equation}\label{eqn:keyfml2a}
\int\limits_{\ball} \frac{(1-|\xi|^2)^{\gamma}dV(\xi)}{|1-
\eta \cdot \overline{\xi}|^{2\theta}}
=\frac{\pi^{n+1}\Gamma(1+\gamma)}{\Gamma(n+2+\gamma)}
\hyperg{\theta}{\theta}{n+2+\gamma}{|\eta|^2}
\end{equation}
holds for all $\eta\in \ball$.
\end{lem}

\begin{proof}
See \cite[Corollary 2.4]{Liu15}.
\end{proof}

\begin{lem}
Let $\gamma>-1$ and $\theta\in \bbR$. Then
\begin{align}\label{eqn:myformula}
\int\limits_{\calU^n} & \frac {\bfrho(w,w)^{\gamma}dV(w)} {|\bfrho(z,w)|^{2\theta} |\bfrho(w,\bfi)|^{2(n+2+\gamma-\theta)}}  \notag \\
&= \frac {4 \pi^{n+1} \Gamma(1+\gamma)}{\Gamma(n+2+\gamma)} \hyperg {\theta} {\theta} {n+2+\gamma} {1-\frac {\bfrho(z,z)}{|\bfrho(\bfi,z)|^2}} |\bfrho(z,\bfi)|^{-2\theta}
\end{align}
holds for all $z\in \calU^n$.
\end{lem}

\begin{proof}
%Again, we transform the integral into an integral on the unit ball $\ball$, via the Cayley transform.
For fixed $z\in \calU^n$, consider the integral
\[
I(z):=\int\limits_{\bbB^{n+1}} \frac {(1-|\xi|^2)^{\gamma}} {|1-\Psi(z)\cdot \bar{\xi}|^{2\theta}} dV(\xi).
\]
Making the change of variables $\xi=\Psi(w)$ and using \eqref{eqn:cayleyidentity1}--\eqref{eqn:jacobian}, we get
\begin{align}\label{eqn:expint1}
I(z) ~=~& \int\limits_{\calU^n} \left(\frac {\bfrho(w,w)}{|\bfrho(w,\bfi)|^2}\right)^{\gamma}
\left|\frac {\bfrho(z,w)}{\bfrho(z,\bfi)\bfrho(\bfi,w)}\right|^{-2\theta} \frac {1}{4|\bfrho(w,\bfi)|^{2(n+2)}} dV(w) \notag \\
=~& \frac {|\bfrho(z,\bfi)|^{2\theta}}{4} \int\limits_{\calU^n} \frac {\bfrho(w,w)^{\gamma}dV(w)} {|\bfrho(z,w)|^{2\theta} |\bfrho(\bfi,w)|^{2(n+2+\gamma-\theta)}}.
\end{align}
On the other hand, by \eqref{eqn:keyfml2a} and \eqref{eqn:cayleyidentity2},
\begin{align}\label{eqn:expint2}
I(z) ~=~& \frac {\pi^{n+1} \Gamma(1+\gamma)} {\Gamma(n+2+\gamma)}
\hyperg {\theta}{\theta}{n+2+\gamma} {\left|\Psi(z)\right|^2} \notag\\
=~& \frac {\pi^{n+1} \Gamma(1+\gamma)} {\Gamma(n+2+\gamma)}
\hyperg {\theta}{\theta}{n+2+\gamma} {1-\frac {\bfrho(z,z)}{|\bfrho(z,\bfi)|^2}}.
\end{align}
The identity \eqref{eqn:myformula} now follows by comparing \eqref{eqn:expint1} and \eqref{eqn:expint2}.
\end{proof}

\begin{lem}
If $\kappa>-n-2$ and $\theta>\max\{\kappa,0\}$, then
\begin{align}\label{eqn:myformula2}
\int\limits_{\calU^n} & \frac {dV(w)} {\bfrho(z,w)^{n+2} \bfrho(\bfi, w)^{\kappa} \bfrho(w,\bfi)^{\theta-\kappa}}  \notag\\
&=~ \frac {4 \pi^{n+1} \Gamma(\theta)}{\Gamma(\theta-\kappa)\Gamma(n+2+\kappa)}
\hyperg {\theta} {\kappa} {n+2+\kappa} {1-\bfrho(z,\bfi)^{-1}} \bfrho(z,\bfi)^{-\theta}
\end{align}
holds for all $z\in \calU^n$.
\end{lem}

\begin{proof}
Making the change of variables $w=\Phi(\xi)$ and using \eqref{eqn:identity14phi} and \eqref{eqn:jacobian4phi},
the integral becomes
\begin{align*}
%\int\limits_{\calU^n} & \frac {dV(w)} {\bfrho(z,w)^{n+2} \bfrho(w,\bfi)^{\theta-\kappa} \bfrho(\bfi, w)^{\kappa}}  \notag\\
&\int\limits_{\ball} \left[\frac {1-\Psi(z)\cdot \overline{\xi}} {(1+[\Psi(z)]_{n+1})(1+\xi_{n+1})}\right]^{-n-2}
(1+\xi_{n+1})^{\theta-\kappa} (1+\overline{\xi}_{n+1})^{\kappa} \frac {4V(\xi)} {|1+\xi_{n+1}|^{2(n+2)}}\\
&\quad =~ 4(1+[\Psi(z)]_{n+1})^{n+2} \int\limits_{\ball}  \frac {dV(\xi)} {(1-\Psi(z)\cdot \overline{\xi})^{n+2}
(1+\overline{\xi}_{n+1})^{-\kappa} (1+\xi_{n+1})^{n+2-\theta+\kappa}}.
\end{align*}
By \eqref{eqn:crucial4}, this equals
\begin{align*}
& \frac {4\pi^{n+1}\Gamma(\theta)} {\Gamma(n+2+\kappa)\Gamma(\theta-\kappa)} (1+[\Psi(z)]_{n+1})^{n+2}
\hyperg {n+2-\theta+\kappa} {n+2} {n+2+\kappa} {-[\Psi(z)]_{n+1}}\\
&\quad =~ \frac {4\pi^{n+1}\Gamma(\theta)} {\Gamma(n+2+\kappa)\Gamma(\theta-\kappa)} (1+[\Psi(z)]_{n+1})^{\theta}
\hyperg {\theta} {\kappa} {n+2+\kappa} {-[\Psi(z)]_{n+1}},
\end{align*}
and \eqref{eqn:myformula2} is proved, in view of that $1+[\Psi(z)]_{n+1} = \bfrho(z,\bfi)^{-1}$.
\end{proof}

%\begin{lem}
%Let $\kappa$ is real and $\theta>\max\{\kappa,0\}$. Then
%\begin{align}\label{eqn:myformula}
%\int\limits_{b\calU^n} & \frac {d\bfbeta(u)} {\bfrho(z,u)^{n+1} \bfrho(u,\bfi)^{\theta-\kappa} \bfrho(\bfi, u)^{\kappa}}  \\
%&= \frac {4 \pi^{n+1} \Gamma(\theta)}{\Gamma(\theta-\kappa)\Gamma(n+1+\kappa)} \frac {1} {\bfrho(z,\bfi)^{\theta}}
%\hyperg {\theta} {\kappa} {n+1+\kappa} {1-\frac {1}{\bfrho(z,\bfi)}} \notag
%\end{align}
%holds for all $z\in \calU^n$.
%\end{lem}

\section{Proof of Theorem \ref{thm:main2}}

\subsection{The upper estimate}
%{\textsc{Proof of the main theorem}}
We first show that
\begin{equation*}
\|T_{\alpha}\|_{p\to p} ~\leq~ \frac {(n+1)! \Gamma\left(1+\alpha-\frac {1}{p}\right)
\Gamma\left(\frac {1}{p}\right)} {\Gamma^2\left(\frac {n+2+\alpha}{2}\right)}.
\end{equation*}
We shall distinguish two cases.
%\begin{itemize}
%\item
%Case I: $p=1$.
%\end{itemize}

\subsubsection*{Case 1: $p=1$.}
In this case, the assumption $p(1+\alpha)>1$ implies that $\alpha>0$.
By Fubini's theorem and \eqref{eqn:keylem2}, we have
\begin{align*}
\|T_{\alpha} f\|_1 ~=~&  \frac {(n+1)!}{4\pi^{n+1}} \int\limits_{\calU^n} \Bigg| \int\limits_{\calU^n}
\frac {\bfrho(w,w)^{\alpha}} {|\bfrho(z,w)|^{n+2+\alpha}} f(w) dV(w)\Bigg| dV(z)\\
~\leq~& \frac {(n+1)!}{4\pi^{n+1}} \int\limits_{\calU^n} |f(w)|\bfrho(w,w)^{\alpha} \Bigg\{ \int\limits_{\calU^n}
\frac {dV(z)} {|\bfrho(z,w)|^{n+2+\alpha}} \Bigg\} dV(w)\\
=~& \frac {(n+1)! \Gamma(\alpha)} {\Gamma^2 (\frac {n+2+\alpha}{2})} \|f\|_1.
\end{align*}

%To this end, we use
%Schur's test \hskip4pt + \hskip4pt the above Corollary.

\subsubsection*{Case 2: $1<p<\infty$.}

The proof appeals to Schur's test (Lemma \ref{lem:schurtest}).
With
\[
K(z,w) = \frac {(n+1)!}{4\pi^{n+1}} \frac {\bfrho(w,w)^{\alpha}} {|\bfrho(z,w)|^{n+2+\alpha}}
\]
and
\[
g(z)=\bfrho(z,z)^{-\frac {1}{pq}},
\]
where $q=p/(p-1)$, using \eqref{eqn:keylem2}, we see that
\begin{align*}
\int\limits_{\calU^n} K(z,w) g(w)^{q} dV(w) ~=~& \frac {(n+1)!}{4\pi^{n+1}}
\int\limits_{\calU^n} \frac {\bfrho(w,w)^{\alpha-\frac {1}{p}}} {|\bfrho(z,w)|^{n+2+\alpha}} dV(w)\\
=~& \frac {(n+1)!  \Gamma\left(1+\alpha-\frac {1}{p}\right) \Gamma\left(\frac {1}{p}\right)}
{\Gamma^2\left(\frac {n+2+\alpha}{2}\right)}\, \bfrho(z,z)^{-\frac {1}{p}}\\
=~& \frac {(n+1)!  \Gamma\left(1+\alpha-\frac {1}{p}\right) \Gamma\left(\frac {1}{p}\right)}
{\Gamma^2\left(\frac {n+2+\alpha}{2}\right)}\, g(z)^{q}
\end{align*}
holds for every $z\in \calU^n$. Similarly,
\begin{align*}
\int\limits_{\calU^n} K(z,w) g(z)^p dV(z) ~=~& \frac {(n+1)!}{4\pi^{n+1}} \bfrho(w,w)^{\alpha}
\int\limits_{\calU^n} \frac {\bfrho(z,z)^{-\frac {1}{q}}} {|\bfrho(z,w)|^{n+2+\alpha}} dV(z)\\
=~& \frac {(n+1)! \Gamma\left(1-\frac {1}{q}\right) \Gamma\left(\alpha+ \frac {1}{q}\right)}
{\Gamma^2\left(\frac {n+2+\alpha}{2}\right)}\, \bfrho(w,w)^{-\frac {1}{q}}\\
=~& \frac {(n+1)!  \Gamma\left(1+\alpha-\frac {1}{p}\right) \Gamma\left(\frac {1}{p}\right)}
{\Gamma^2\left(\frac {n+2+\alpha}{2}\right)}\, g(w)^{p}
\end{align*}
holds for every $w\in \calU^n$. Thus, by Schur's test, this yields the desired upper bound.

\subsection{The lower estimate}

We now proceed to show
\begin{equation*}
\|T_{\alpha}\|_{p\to p} ~\geq~ \frac {(n+1)! \Gamma\left(1+\alpha-\frac {1}{p}\right)
\Gamma\left(\frac {1}{p}\right)} {\Gamma^2\left(\frac {n+2+\alpha}{2}\right)}.
\end{equation*}
%We only need to consider the case when $p\geq 2$,
%and the case when $1<p<2$ then follows from the duality.

\subsubsection*{Case I: $(n+2+\alpha)p > n+3$.}

For notational convenience, we write $\beta:=\frac {n+2+\alpha}{2}$.
For $0<t<\frac {1}{p}$, we consider the function
\[
\psi_{t}(w):= \frac {\bfrho(w,w)^{-t}} {|\bfrho(w,\bfi)|^{2(\beta-t)}}, \quad w\in \calU^n.
\]
Note that the assumption  $(n+2+\alpha)p > n+3$ implies that
$(2\beta-t)p-n-2>0$, which guarantees that $\psi_{t}\in L^p(\calU^n)$ for all $t\in (0,\frac {1}{p})$. Indeed,
by applying \eqref{eqn:keylem2} with $\gamma=-tp$ and $\theta=(2\beta-t)p-n-2$, we have
\begin{equation*} %\label{eqn:normofft}
\|\psi_{t}\|_p^{p} ~=~ \frac {4\pi^{n+1} \Gamma(1-t p) \Gamma((2\beta-t)p-n-2)} {\Gamma^2 \left((\beta-t)p\right)}.
\end{equation*}
This implies that
\begin{equation}\label{eqn:normofft}
\lim_{t\nearrow \frac {1}{p}} \|\psi_{t}\|_p^{-p}=0,
\end{equation}
since $\Gamma(1-t p)\to \infty$ as $t\nearrow \frac {1}{p}$.

Next, applying \eqref{eqn:myformula} with $\gamma=2\beta-t-n-2$ and $\theta=\beta$, we obtain
\begin{align*}
(T_{\alpha} \psi_{t}) (z) ~=~& \frac {(n+1)!}{4\pi^{n+1}} \int\limits_{\calU^{n}} \frac {\bfrho(w,w)^{2\beta-t-n-2}} {|\bfrho(z,w)|^{2\beta}
|\bfrho(w,\bfi)|^{2(\beta-t)}} dV(w)\\
=~& \frac {(n+1)! \Gamma(2\beta-t-n-1)} {\Gamma(2\beta-t)}\, \hyperg {\beta}
{\beta} {2\beta-t} {1-\frac {\bfrho(z,z)}{|\bfrho(\bfi,z)|^2}} \, |\bfrho(\bfi,z)|^{-2\beta}\\
=~& \frac {(n+1)! \Gamma(2\beta-t-n-1)} {\Gamma(2\beta-t)}\,\\
&\quad \times
\hyperg {\beta-t} {\beta-t} {2\beta-t} {1-\frac {\bfrho(z,z)}{|\bfrho(\bfi,z)|^2}} \frac {\bfrho(z,z)^{-t}} {|\bfrho(z,\bfi)|^{2(\beta-t)}}.
\end{align*}
The last equality follows from \eqref{eqn:euler}.
%\Begin{Align*}
%T_{\Alpha} F_{T} (Z) ~=~& \Int\Limits_{\Calu^{N}} \Frac {\Rho(W,W)^{\Beta-N-2-T}} {|\Rho(Z,W)|^{2\Beta}
%|\Rho(W,\Bfi)|^{2\Beta-2T}} Dv(W)\\
%=~& \Frac {4 \Pi^{N+1} \Gamma(\Beta-N-1-T)} {\Gamma(2\Beta-T)}\, |\Rho(\Bfi,Z)|^{-2\Beta} \Hyperg {\Beta}
%{\Beta} {2\Beta-T} {1-\Frac {\Rho(Z,Z)}{|\Rho(\Bfi,Z)|^2}}\\
%=~& \Frac {4 \Pi^{N+1} \Gamma(\Beta-N-1-T)} {\Gamma(2\Beta-T)}\, |\Rho(\Bfi,Z)|^{-2\Beta}
%\Left(\Frac {\Rho(Z,Z)}{|\Rho(\Bfi,Z)|^2} \Right)^{-T} \\
%& \Quad \Times  \Hyperg {\Beta-T} {\Beta-T} {2\Beta-T} {1-\Frac {\Rho(Z,Z)}{|\Rho(\Bfi,Z)|^2}}\\
%=~& \Frac {4 \Pi^{N+1} \Gamma(1+\Alpha-T)} {\Gamma(N+2+\Alpha-T)}\, H\Left(T, 1-\Frac {\Rho(Z,Z)}{|\Rho(\Bfi,Z)|^2}\Right) F_T(Z),
%\End{Align*}
%where
%\[
%H(t,\lambda):=\hyperg {\frac {n+2+\alpha}{2}-t}
%{\frac {n+2+\alpha}{2}-t} {n+2+\alpha-t} {\lambda}.
%\]
For simplicity, we rewrite this as
\begin{equation}\label{eqn:Tf-2ndexprn}
(T_{\alpha}\psi_{t})(z)=\frac {(n+1)! \Gamma(2\beta-t-n-1)\Gamma(t)} {\Gamma^2(\beta)}\,
H\left(t, 1-\frac {\bfrho(z,z)}{|\bfrho(\bfi,z)|^2}\right) \psi_{t}(z),
\end{equation}
with
\[
H(t,\lambda)~:=~ \frac {\Gamma^2(\beta)} {\Gamma(2\beta-t)\Gamma(t)}\,
\hyperg {\beta-t} {\beta-t} {2\beta-t} {\lambda}.
\]
Note that the above hypergeometric function is increasing on the interval $[0,1)$,
since its Taylor coefficients are all positive.

%So, for each fixed $t\in (0,\frac {1}{p})$,
%the function $\lambda \mapsto H(t, \lambda)$ is increasing on $[0,1)$. Moreover,
%it tends to
%\[
%H^{\ast}(t)=\frac {4\pi^{n+1}\Gamma(\beta-t-n-1) \Gamma(t)} {\Gamma^2(\beta)};
%\]
%as $\lambda \nearrow 1$.

%\begin{itemize}
%\item
%for each fixed $t\in (0,\frac {1}{p})$, the function $\lambda \mapsto H(t, \lambda)$ is nondecreasing on the interval $[0,1)$;
%\item
%for each fixed $t\in (0,\frac {1}{p})$,
%\[
%\lim_{\lambda\nearrow 1} H(t,\lambda)=\frac {4\pi^{n+1}\Gamma(\beta-t-n-1) \Gamma(t)} {\Gamma^2(\beta)};
%\]
%\item
%the above convergence is uniform in $t$ on $(0,\frac {1}{p})$, by Dini's theorem.
%\end{itemize}
%
%We first note that the hypergeometric function above is increasing on the interval $[0,1)$,
%since its Taylor coefficients are all positive.

Now we think of $H(t,\lambda)$ as a family of continuous functions of $t$ on $[\frac {1}{2p},\frac {1}{p}]$ indexed by $\lambda\in [0,1)$.
These functions tend monotonically to the constant function $1$ pointwise as $\lambda \nearrow 1$,
by \eqref{eqn:gauss}.  Moreover, the convergence is uniform in $t$ on $[\frac {1}{2p},\frac {1}{p}]$, by Dini's theorem.

It follows that for any $\epsilon > 0$, there exists a $\delta>0$, sufficiently small and independent of $t\in [\frac {1}{2p},\frac {1}{p}]$, such that
\[
H(t,\lambda)\geq 1-\epsilon, \quad \text{provided } \lambda>1-\delta.
\]
This, together with \eqref{eqn:Tf-2ndexprn}, shows that
\begin{equation}\label{eqn:Tf-pointwiseestimate}
|(T_{\alpha} \psi_{t}) (z)| ~\geq~ (1-\epsilon) \frac {(n+1)!\Gamma(2\beta-t-n-1) \Gamma(t)} {\Gamma^2(\beta)}
\chi_{E_{\delta}}(z)|\psi_{t}(z)|
\end{equation}
holds for all $z\in \calU^n$ and all $t\in [\frac {1}{2p},\frac {1}{p})$, where
\begin{equation}\label{eqn:setEdelta}
E_{\delta}:=\left\{ z\in \calU: \frac {\bfrho(z,z)}{|\bfrho(\bfi,z)|^2} <\delta \right\},
\end{equation}
and $\chi_{E_{\delta}}$ denotes the indicator function for the set $E_{\delta}$.
Consequently,
%Since $\|T_{\alpha}\|_{p\to p} \geq  \|T_{\alpha} \psi_{t}\|_p /\|\psi_{t}\|_p$, it follows from \eqref{eqn:Tf-pointwiseestimate} that
%\begin{align}\label{eqn:finalestimate}
%\|T_{\alpha}\|_{p\to p} ~\geq~&  (1-\epsilon) \frac {4\pi^{n+1}\Gamma(2\beta-t-n-1) \Gamma(t)} {\Gamma^2(\beta)}\\
%& \qquad \qquad \times \Bigg\{ 1 - \|\psi_{t}\|_p^{-p} \int\limits_{\calU^n\setminus E_{\delta}} |\psi_{t}|^p dV  \Bigg\}^{\frac {1}{p}}, \notag
%\end{align}
\begin{align}\label{eqn:finalestimate}
\|T_{\alpha}\|_{p\to p} ~\geq~&  (1-\epsilon) \frac {(n+1)!\Gamma(2\beta-t-n-1) \Gamma(t)} {\Gamma^2(\beta)}
\Bigg\{ 1 - \|\psi_{t}\|_p^{-p} \int\limits_{\calU^n\setminus E_{\delta}} |\psi_{t}|^p dV  \Bigg\}^{\frac {1}{p}},
\end{align}
since $\|T_{\alpha}\|_{p\to p} \geq  \|T_{\alpha} \psi_{t}\|_p /\|\psi_{t}\|_p$.
%
%Note that
%\begin{align*}
%\|\psi_{t}\|_p^{-p} ~=~& \Bigg\{\int\limits_{\calU^n} \frac {\bfrho(w,w)^{-t p}}
%{|\bfrho(\bfi,w)|^{(n+1+\alpha-2t)p}} dV(w)\Bigg\}^{-1} \\
%=~& \frac {\Gamma^2 \left((n+1+\alpha-2t)p/2\right)} {4\pi^n \Gamma(1-t p) \Gamma((n+1+\alpha-t)p-n-1)}
%~\to~ 0
%\end{align*}
%as $t\nearrow \frac {1}{p}$.

Making the change of variables $w=\Phi(\xi)$ and using \eqref{eqn:cayleyidentity2} and \eqref{eqn:jacobian},
we have
\begin{align*}
\int\limits_{\calU^n\setminus E_{\delta}} |\psi_{t}|^p dV ~=~& \int\limits_{\calU^n\setminus E_{\delta}}  \frac {\bfrho(w,w)^{-tp}} {|\bfrho(\bfi,w)|^{2(\beta-t)p}} dV(w)\\
=~& \int\limits_{|\xi|\leq \sqrt{1-\delta}} \left(\frac {1-|\xi|^2} {|1+\xi_{n+1}|^2} \right)^{-tp}
|1+\xi_{n+1}|^{2(\beta-t)p} \frac {4 dV(\xi)}{|1+\xi_{n+1}|^{2(n+2)}}\\
=~& 4 \int\limits_{|\xi|\leq \sqrt{1-\delta}} \frac {(1-|\xi|^2)^{-tp}}
{|1+\xi_{n+1}|^{2(n+2-\beta p)}} dV(\xi).
%\leq~& 2^{n+3+\alpha-2t} \left( \frac {1-|\eta|^2}{|1-\eta_1|}\right)^{2t-n-1-\alpha}
%\int\limits_{|\xi|\leq \sqrt{1-\delta}} \frac {(1-|\xi|^2)^{-t} dV(\xi)} {|1-\xi_1|^{n+1-\alpha+t}}
\end{align*}
It is easily seen that if $|\xi| \leq \sqrt{1-\delta}$ then
\[
|1+\xi_{n+1}|^{2(\beta p-n-2)} \leq \max\left\{ \left(\frac{2}{\delta}\right)^{2(n+2-\beta p)},
2^{2(\beta p-n-2)} \right\}
\]
and
\[
(1-|\xi|^2)^{-tp} \leq \delta^{-tp} \leq \delta^{-1},
\]
hold for all $t\in [\frac {1}{2p},\frac {1}{p})$. It follows that
\[
\sup_{t\in [\frac {1}{2p},\frac {1}{p})} \int\limits_{\calU^n\setminus E_{\delta}} |\psi_{t}|^p dV \leq  C_{\delta},
\]
with
\begin{align*}
C_{\delta} ~:=~& \frac {2\pi^{n+1}} {(n+1)!} \frac {(1-\delta)^{n+1}} {\delta}
\max\left\{ \left(\frac{2}{\delta}\right)^{2(n+2-\beta p)},
2^{2(\beta p-n-2)} \right\}.
\end{align*}
Keep in mind that $\delta$ is independent of $t\in [\frac {1}{2p},\frac {1}{p})$. Combining this with \eqref{eqn:normofft}, we obtain
\[
\lim_{t\nearrow \frac {1}{p}} \|\psi_{t}\|_p^{-p} \int\limits_{\calU^n\setminus E_{\delta}} |\psi_{t}|^p dV =0.
\]

Now, letting $t\nearrow \frac {1}{p}$ in \eqref{eqn:finalestimate}, we conclude that
\begin{align*}
\|T_{\alpha}\|_{p\to p} ~\geq~&  (1-\epsilon) \frac {(n+1)! \Gamma\left(2\beta-n-1-\frac {1}{p}\right)
\Gamma\left(\frac {1}{p}\right)} {\Gamma^2(\beta)}\\
~=~& (1-\epsilon) \frac {(n+1)! \Gamma\left(1+\alpha-\frac {1}{p}\right)\Gamma\left(\frac {1}{p}\right)}
{\Gamma^2\left(\frac {n+2+\alpha}{2}\right)}.
\end{align*}
Since $\epsilon$ is arbitrary, this yields
\[
\|T_{\alpha}\|_{p\to p} ~\geq~  \frac {(n+1)! \Gamma\left(1+\alpha-\frac {1}{p}\right)\Gamma\left(\frac {1}{p}\right)}
{\Gamma^2\left(\frac {n+2+\alpha}{2}\right)}.
\]

\subsubsection*{Case II: $(n+2+\alpha)p \leq n+3$.}

In this case, the above test function $\psi_t$ does not belong to $L^p(\calU^n)$.
Instead, for $t\in \left(\max\{0,-\alpha\},\frac {1}{q}\right)$, we consider the function
\[
\tilde{\psi}_{t}(w):= \frac {\bfrho(w,w)^{-t}} {|\bfrho(w,\bfi)|^{2(\tilde{\beta}-t)}}, \quad w\in \calU^n,
\]
where
\[
q:=\frac {p}{p-1} \quad \text{ and }\quad  \tilde{\beta}:=\frac {n+2-\alpha}{2}.
\]
Note that the assumption  $(n+2+\alpha)p \leq  n+3$ implies
$(2\tilde{\beta}-t)q-n-2>0$, which guarantees that $\tilde{\psi}_{t}\in L^q(\calU^n)$ for all $t\in (0,\frac {1}{q})$. Indeed,
by applying \eqref{eqn:keylem2} with $\gamma=-tq$ and $\theta=(2\tilde{\beta}-t)q-n-2$, we have
\begin{equation*} %\label{eqn:normofft}
\|\tilde{\psi}_{t}\|_q^{q} ~=~ \frac {4\pi^{n+1} \Gamma(1-t q) \Gamma((2\tilde{\beta}-t)q-n-2)}
{\Gamma^2 ((\tilde{\beta}-t)q)}.
\end{equation*}
This also implies that
\begin{equation}\label{eqn:normofft}
\lim_{t\nearrow \frac {1}{q}} \|\tilde{\psi}_{t}\|_q^{-q}=0,
\end{equation}
since $\Gamma(1-t q)\to \infty$ as $t\nearrow \frac {1}{q}$.

Note that the adjoint of $T_{\alpha}$ is given by
\[
(T_{\alpha}^{\ast} f) (z) = \frac {(n+1)!} {4\pi^{n+1}} \bfrho(z,z)^{\alpha}
\int_{\calU^n} \frac {f(w)} {|\bfrho(z,w)|^{n+2+\alpha}} dV(w).
\]
Applying \eqref{eqn:myformula} with $\gamma=-t$ and $\theta=n+2-\tilde{\beta}$, we obtain
\begin{align*}
(T_{\alpha}^{\ast} \tilde{\psi}_{t}) (z) ~=~& \frac {(n+1)!}{4\pi^{n+1}} \bfrho(z,z)^{n+2-2\tilde{\beta}}
\int\limits_{\calU^{n}} \frac {\bfrho(w,w)^{-t}} {|\bfrho(z,w)|^{2(n+2-\tilde{\beta})}
|\bfrho(w,\bfi)|^{2(\tilde{\beta}-t)}} dV(w)\\
=~& \frac {(n+1)! \Gamma(1-t)} {\Gamma(n+2-t)}\,
\hyperg {\tilde{\beta}-t} {\tilde{\beta}-t} {n+2-t} {1-\frac {\bfrho(z,z)}{|\bfrho(\bfi,z)|^2}} \frac {\bfrho(z,z)^{-t}} {|\bfrho(z,\bfi)|^{2(\tilde{\beta}-t)}}.
\end{align*}

A similar argument as in Case I shows that for any $\epsilon > 0$, there exists a $\delta>0$ such that
\begin{equation*}
|(T_{\alpha}^{\ast} \tilde{\psi}_{t}) (z)| ~\geq~ (1-\epsilon) \frac {(n+1)!\Gamma(n+2+t-2\tilde{\beta}) \Gamma(1-t)}
{\Gamma^2(n+2-\tilde{\beta})}
\chi_{E_{\delta}}(z)|\tilde{\psi}_{t}(z)|
\end{equation*}
holds for all $z\in \calU^n$ and all $t\in \left[\max\{-\alpha,\frac {1}{2q}\},\frac {1}{q}\right)$, where $E_{\delta}$
is as in \eqref{eqn:setEdelta} and $\chi_{E_{\delta}}$ denotes the indicator function for the set $E_{\delta}$.
Consequently,
\begin{align}\label{eqn:finalestimate2}
\|T_{\alpha}^{\ast}\|_{q\to q} ~\geq~&  (1-\epsilon) \frac {(n+1)!\Gamma(n+2-2\tilde{\beta}+t) \Gamma(1-t)}
{\Gamma^2(n+2-\tilde{\beta})} \notag \\
&\qquad \times \Bigg\{ 1 - \|\tilde{\psi}_{t}\|_q^{-q} \int\limits_{\calU^n\setminus E_{\delta}} |\tilde{\psi}_{t}|^q dV  \Bigg\}^{\frac {1}{q}}.
\end{align}
Exactly as in Case I, we can let $t\nearrow \frac {1}{q}$ in \eqref{eqn:finalestimate2} to yield
\begin{align*}
\|T_{\alpha}^{\ast}\|_{q\to q} ~\geq~&  (1-\epsilon) \frac {(n+1)! \Gamma\left(n+2-2\tilde{\beta}+\frac {1}{q}\right)
\Gamma\left(1-\frac {1}{q}\right)} {\Gamma^2(n+2-\tilde{\beta})}\\
~=~& (1-\epsilon) \frac {(n+1)! \Gamma\left(1+\alpha-\frac {1}{p}\right)\Gamma\left(\frac {1}{p}\right)}
{\Gamma^2\left(\frac {n+2+\alpha}{2}\right)}.
\end{align*}
Since $\epsilon$ is arbitrary, this yields
\[
\|T_{\alpha}\|_{p\to p} ~=~ \|T_{\alpha}^{\ast}\|_{q\to q} ~\geq~  \frac {(n+1)! \Gamma\left(1+\alpha-\frac {1}{p}\right)\Gamma\left(\frac {1}{p}\right)}
{\Gamma^2\left(\frac {n+2+\alpha}{2}\right)}.
\]
The proof is now complete.

\section{Proof of Theorem \ref{thm:main0}}

As is mentioned in the introduction, the second inequality in \eqref{eqn:bergmanestimate}
follows immediately from Theorem \ref{thm:main2},
so we prove only the first inequality.

Further, we only need to consider the case when $p\geq 2$, and
the case when $1<p<2$ then follows from the duality.

For notational convenience, we put
\[
\kappa:=(n+2)\left(\frac {1}{2} - \frac {1}{p}\right) \quad \text{and} \quad
\theta:=\frac {n+2}{p}.
\]
Note that $\theta+\kappa=\frac {n+2}{2}$.

For $0<\epsilon<\kappa$, we consider the function
\[
f_{\epsilon}(z) := \frac {1}{\bfrho(\bfi,z)^{\epsilon-\kappa}  \bfrho(z,\bfi)^{\theta+\kappa}}, \quad z\in \calU^n.
\]
Using \eqref{eqn:myformula2} we get
\begin{align*}
\left(P_{\calU^n} f_{\epsilon}\right) (z) =& \frac {(n+1)!}{4\pi^{n+1}}\int\limits_{\calU^n} \frac {dV(w)}
{\bfrho(z,w)^{n+2} \bfrho(\bfi, w)^{\epsilon-\kappa} \bfrho(w,\bfi)^{\theta+\kappa}}  \notag\\
=& \frac {(n+1)!\, \Gamma(\theta+\epsilon)}{\Gamma(\theta+\kappa)\Gamma(n+2+\epsilon-\kappa)}
\hyperg {\epsilon-\kappa} {\theta+\epsilon} {n+2+\epsilon-\kappa} {1-\frac {1}{\bfrho(z,\bfi)}}
\frac {1}{\bfrho(z,\bfi)^{\theta+\epsilon}}.
\end{align*}
Now we consider the functions
\begin{align*}
g_{\epsilon}(z) ~:=~& \frac {(n+1)! \Gamma(\theta+\epsilon)}{\Gamma(\theta+\kappa)\Gamma(n+2+\epsilon-\kappa)}
\hyperg {\epsilon-\kappa} {\theta+\epsilon} {n+2+\epsilon-\kappa} {1} \frac {1}{\bfrho(z,\bfi)^{\theta+\epsilon}}, \\
=~& \frac {\Gamma(\theta+\epsilon)\Gamma(n+2-\theta-\epsilon)} {\Gamma(\theta+\kappa)\Gamma(n+2-\kappa-\theta)}
\, \frac {1}{\bfrho(z,\bfi)^{\theta+\epsilon}}, \\
\intertext{and}
h_{\epsilon}(z) ~:=~& \frac {(n+1)! \Gamma(\theta+\epsilon)}{\Gamma(\theta+\kappa)\Gamma(n+2+\epsilon-\kappa)}\,
\frac {1}{\bfrho(z,\bfi)^{\theta+\epsilon}}, \\
& \quad \times \left\{\hyperg {\epsilon-\kappa} {\theta+\epsilon} {n+2+\epsilon-\kappa} {1-\frac {1}{\bfrho(z,\bfi)}}
- \hyperg {\epsilon-\kappa} {\theta+\epsilon} {n+2+\epsilon-\kappa} {1} \right\}.
\end{align*}
It is clear that $P_{\calU^n}f_{\epsilon}=g_{\epsilon} + h_{\epsilon}$, and hence
\begin{equation*}
\|P_{\calU^n}\|_p ~\geq~ \limsup_{\epsilon\to 0^{+}}\frac {\|P_{\calU^n}f_{\epsilon}\|_p} {\|f_{\epsilon}\|_p}
~\geq~ \limsup_{\epsilon\to 0^{+}} \left(\frac {\|g_{\epsilon}\|_p} {\|f_{\epsilon}\|_p}
- \frac {\|h_{\epsilon}\|_p} {\|f_{\epsilon}\|_p} \right).
\end{equation*}
It is clear that
\[
\|g_{\epsilon}\|_p = \frac {\Gamma(\theta+\epsilon)\Gamma(n+2-\theta-\epsilon)}
{\Gamma(\theta+\kappa)\Gamma(n+2-\kappa-\theta)} \|f_{\epsilon}\|_p,
\]
and hence
\begin{align*}
\lim_{\epsilon\to 0^{+}} \frac {\|g_{\epsilon}\|_p} {\|f_{\epsilon}\|_p}
~=~& \frac {\Gamma(\theta)\Gamma(n+2-\theta)} {\Gamma(\theta+\kappa)\Gamma(n+2-\kappa-\theta)}\\
~=~& \frac {\Gamma\left(\frac {n+2}{p}\right)\Gamma\left(n+2-\frac {n+2}{p}\right)}
{\Gamma^2\left(\frac {n+2}{2}\right)}.
\end{align*}
Thus we are reduced to proving
\begin{equation*}
\limsup_{\epsilon\to 0^{+}} \frac {\|h_{\epsilon}\|_p} {\|f_{\epsilon}\|_p} ~=~ 0.
\end{equation*}
Using \eqref{eqn:keylem2}, we see that
\[
\|f_{\epsilon}\|_p^p ~=~ \int\limits_{\calU^n} \frac {dV(w)} {|\bfrho(\bfi,w)|^{n+2+p\epsilon}}
~=~ \frac {4\pi^{n+1} \Gamma(p\epsilon)} {\Gamma^2\left(\frac {n+2+p\epsilon}{2}\right)} ~\to~ \infty \quad \text{ as } \epsilon\to 0^{+}.
\]
Hence, the proof is completed by showing that
\begin{equation}\label{eqn:unifbdd}
\sup\limits_{0<\epsilon<\kappa} \|h_{\epsilon}\|_p < \infty.
\end{equation}
Indeed, by Lemma \ref{lem:est4hyperg}, it is easy to check that
%\begin{align*}
%& \left|\hyperg {\kappa+\epsilon} {\theta+\epsilon} {n+2+\kappa+\epsilon} {1-\frac {1}{\bfrho(z,\bfi)}}
%- \hyperg {\kappa+\epsilon} {\theta+\epsilon} {n+2+\kappa+\epsilon} {1} \right|\\
%& \quad \leq~ |\kappa+\epsilon|\cdot \max \left\{ 2^{-\kappa-\epsilon-1}, \frac {\Gamma(n+2+\kappa+\epsilon)
%\Gamma(n+1-\theta-\epsilon)} {\Gamma(n+1)\Gamma(n+2+\kappa-\theta)} \right\}  \frac {1}{|\bfrho(z,\bfi)|}.
%\end{align*}
\[
|h_{\epsilon}(z)| ~\leq~ \frac {C(\epsilon)} {|\bfrho(z,\bfi)|^{\theta+1+\epsilon}}
\]
holds for all $z\in \calU^n$, with the constant
\begin{align*}
C(\epsilon) ~:=~& \frac {|\epsilon-\kappa|\Gamma(\theta+\epsilon)}{\Gamma(\theta+\kappa)}
\max \left\{ \frac {2^{\kappa-\epsilon-1} (n+1)!}
{\Gamma(n+2+\epsilon-\kappa)}, \frac {(n+1) \Gamma(n+1-\theta-\epsilon)}
{\Gamma(n+2-\kappa-\theta)} \right\}.
%=~& \frac {\left|(n+2)\left(\frac {1}{p} - \frac {1}{2}\right)+\epsilon\right|
%\Gamma\left(\frac {n+2}{p}+\epsilon\right)} {\Gamma\left(\frac {n+2}{2}\right)}\\
%& \quad \times\, \max \Bigg\{ \frac {2^{-(n+2)\left(\frac {1}{p} - \frac {1}{2}\right)+\epsilon-1} (n+1)!}
%{\Gamma\left((n+2)\left(\frac {1}{p}+\frac {1}{2}\right)+\epsilon\right)},\,
%\frac {(n+1) \Gamma\left(n+1-\frac {n+2}{p}-\epsilon\right)}
%{\Gamma\left(\frac {n+2}{2}\right)} \Bigg\}.
\end{align*}
Therefore, by \eqref{eqn:keylem2}, we obtain
\begin{align*}
\|h_{\epsilon}\|_p^p ~\leq~ C(\epsilon)^p \int\limits_{\calU^n} \frac {dV(w)} {|\bfrho(w,\bfi)|^{n+2+p(1+\epsilon)}}
~=~ C(\epsilon)^p \cdot \frac {4\pi^{n+1} \Gamma(p(1+\epsilon))} {\Gamma^2\left(\frac {n+2+p(1+\epsilon)}{2}\right)},
\end{align*}
%Note that $\sup\limits_{0<\epsilon<\kappa} C(\epsilon) <\infty$.
and \eqref{eqn:unifbdd} easily follows. The proof is now complete.

\section{Proof of Theorem \ref{thm:twonorms}}

We first present two auxiliary lemmas, which are interesting for their own sake.

\begin{lem}\label{prop:psh}
If $F$ is plurisubharmonic in $\calU^n$ ($n\geq 1$) and $F\geq 0$, then
\begin{equation}\label{eqn:est4psh}
\pi \int\limits_{\calU^{n-1}} F(0,w) dm_{2n} (w) ~\leq~  \sup_{t>0} \int\limits_{b\calU^{n}} F(z+t\bfi) d\bfbeta(z).
\end{equation}
\end{lem}

\begin{proof}
Note that $(\sqrt{\bfrho_{n-1}(w)} \, e^{i\theta}, w)\in b\calU^n$ for any $w\in \calU^{n-1}$ and any $\theta\in [0,2\pi)$. For $t>0$,
\[
F((0,w)+t\bfi) ~\leq~ \frac {1}{2\pi} \int\limits_{0}^{2\pi} F((\sqrt{\bfrho_{n-1}(w)}\, e^{i\theta}, w)+t\bfi) d\theta,
\]
since $f$ is plurisubharmonic. By Lemma \eqref{eqn:forellispcase1}, it follows that
\begin{align}\label{eqn:tmpno}
\pi \int\limits_{\calU^{n-1}} F(& (0,w) + t\bfi) dm_{2n}(w) \\
& \qquad ~\leq~ \pi \int\limits_{\calU^{n-1}} \Bigg\{ \frac {1}{2\pi}
\int\limits_{0}^{2\pi} F((\sqrt{\bfrho_{n-1}(w)}\, e^{i\theta}, w)+t\bfi) d\theta \Bigg\}  dm_{2n}(w) \notag\\
& \qquad ~=~ \int\limits_{b\calU^{n}} F(z+t\bfi) d\bfbeta(z). \notag
\end{align}
Note that
\[
\int\limits_{\calU^{n-1}} F((0,w)+t\bfi) dm_{2n}(w) ~=~ \int\limits_{\{\bfrho_{n-1}(w)>t\}} F(0,w) dm_{2n}(w).
\]
Hence \eqref{eqn:est4psh} follows from \eqref{eqn:tmpno} as $t\searrow 0$.
\end{proof}

Let $f$ and $g$ be functions with domains $\calU^n$ and $\calU^{n-1}$, respectively.
We define a restriction operator $\mathfrak{R}$ and an extension operator $\mathfrak{E}$ by
\begin{align*}
(\mathfrak{R} f)(\tilde{z}) ~=~& f(0,\tilde{z}) \qquad (\tilde{z}\in \calU^{n-1}),\\
(\mathfrak{E} g)(z_1,\tilde{z}) ~=~ & g(\tilde{z}) \qquad (z=(z_1,\tilde{z})\in \calU^{n}).
\end{align*}

\begin{lem}\label{thm:embedding}
Suppose $n\geq 1$ and $0 < p < \infty$. Then the restriction operator $\mathfrak{R}$ maps $H^p(\calU^n)$
boundedly onto $A^p(\calU^{n-1})$, with operator norm equal to $\pi^{-\frac {1}{p}}$.
\end{lem}

\begin{proof}
Let $f\in H^p(\calU^n)$. Applying Proposition \ref{prop:psh} to $F=|f|^p$, we obtain
\[
 \|\mathfrak{R}f\|_{A^p(\calU^{n-1})}^p  ~\leq~  \pi^{-1}\, \|f\|_{H^p(\calU^n)}^p.
\]
This means that $\mathfrak{R}$ maps $H^p(\calU^n)$
boundedly into $A^p(\calU^{n-1})$, with operator norm at most $\pi^{-\frac {1}{p}}$.

To see that $\mathfrak{R}$ maps $H^p(\calU^n)$ onto $A^p(\calU^{n-1})$, let $f\in A^p(\calU^{n-1})$.
Since $f=\mathfrak{R}\mathfrak{E}f$,  it suffices for us to show that $\mathfrak{E}f$ belongs to $H^p(\calU^n)$.
For any $t>0$, again by the identity \eqref{eqn:forellispcase1}, we have
\begin{align*}
\int\limits_{b\calU^n} & |\mathfrak{E}f  (z+t\bfi)|^p d\bfbeta(z) \\
&\quad ~=~ \pi \int\limits_{\calU^{n-1}} \Bigg\{ \frac {1}{2\pi}
\int\limits_{0}^{2\pi} \left|\mathfrak{E} f((\sqrt{\bfrho_{n-1}(w)}\, e^{i\theta}, w)+t\bfi)\right|^p d\theta \Bigg\}  dm_{2n}(w)\\
&\quad ~=~ \pi \int\limits_{\calU^{n-1}}  |f( w+t\tilde{\bfi})|^p  dm_{2n}(w)
\qquad \qquad \qquad (\tilde{\bfi}=(\underbrace{0,\ldots, 0}_{n-1},i)) \\
&\quad ~=~ \pi \int\limits_{\{\bfrho_{n-1}(w)>t\}}  |f(w)|^p  dm_{2n}(w)  ~\leq ~ \pi \, \|f\|_{A^p(\calU^{n-1})},
\end{align*}
which completes the proof.
\end{proof}

We now turn to the proof of Theorem \ref{thm:twonorms}.

\begin{proof}[Proof of Theorem \ref{thm:twonorms}]

Let $f\in L^p(\calU^{n-1})$ be arbitrary, with $\|f\|_{L^p(\calU^{n-1})}=1$.
Put
\[
g :=\pi^{-\frac {1}{p}}\, \mathfrak{E}  f
\]

By \eqref{eqn:prjformula2}, we see that $g\in L^p(b\calU^n)$ and
\[
\|g\|_{L^p(b\calU^n)} = \|f\|_{L^p(\calU^{n-1})}=1.
\]
%
%For fixed $\tilde{z}\in \calU^{n-1}$, we consider the function
%$h_{\tilde{z}}  := f\cdot K_{\calU^{n-1}} (\tilde{z},\cdot)$,

Note also that
\begin{align*}
S_{\calU^{n}} ((0,\tilde{z}),u) \, g(u) ~=~ \pi^{-1-\frac {1}{p}} K_{\calU^{n-1}} (\tilde{z},\tilde{u}) f(\tilde{u})
\end{align*}
%\begin{align*}
%\pi^{-1-\frac {1}{p}} h_{\tilde{z}} (\Pi (u)) ~=~&
%\pi^{-1-\frac {1}{p}} f(\Pi(u)) K_{\calU^{n-1}} (\tilde{z},\Pi(u))
%\\
% ~=~& S_{\calU^{n}} ((0,\tilde{z}),u) \, g(u)
%\end{align*}
holds for any $\tilde{z}\in \calU^{n-1}$ and $u=(u_1,\tilde{u})\in b\calU^n$, where $K_{\calU^{n-1}} (\cdot,\cdot)$
is the Bergman kernel for $\calU^{n-1}$ and $S_{\calU^n}(\cdot,\cdot)$ is the Cauchy-Szeg\"o kernel for $\calU^n$.
It follows that
\begin{align*}
(C_{\calU^n}g)((0,\tilde{z})) ~=~& \pi^{-1-\frac {1}{p}} \int\limits_{b\calU^n}
K_{\calU^{n-1}} (\tilde{z},\tilde{u}) f(\tilde{u})\, d\bfbeta(u)\\
=~& \pi^{-\frac {1}{p}} \int\limits_{\calU^{n-1}} K_{\calU^{n-1}} (\tilde{z},w) f(w) dm_{2n}(w)\\
~=~& \pi^{-\frac {1}{p}}\, (P_{\calU^{n-1}}f)(\tilde{z}),
\end{align*}
where in the second equality we used \eqref{eqn:prjformula2}. Thus, by Theorem \ref{thm:embedding},
\[
\|P_{\calU^{n-1}}f\|_{A^p(\calU^{n-1})} ~=~ \pi\, \|\mathfrak{R} (C_{\calU^n}g)\|_{A^p(\calU^{n-1})}
~\leq~ \|C_{\calU^n}g\|_{H^p(\calU^{n})}.
\]
Taking supremum over all $f\in L^p(\calU^{n-1})$ with $\|f\|_{L^p(\calU^{n-1})}=1$ yields \eqref{eqn:normcomparison}.
\end{proof}

\end{document}